\documentclass[10pt,oneside,a4paper]{amsart}
\usepackage{geometry}
\usepackage[english]{babel}
\usepackage{amsfonts}
\usepackage{amsmath}
\usepackage{amssymb}
\usepackage{graphicx}
\usepackage{caption}
\usepackage{subcaption}
\usepackage{amsthm}
\usepackage{eurosym}
\usepackage[dvipsnames,usenames]{color}
\usepackage{lmodern} 
\usepackage[T1]{fontenc}
\title{\textbf{Discrete uncertainty principles and Virial identities}}
\author{Aingeru Fern\'andez-Bertolin}
\date{\today}
\address{A. Fern\'andez-Bertolin:Departamento de Matem\'aticas,  Universidad del Pa\'is Vasco UPV/EHU, apartado 644, 48080, Bilbao, Spain}
\email{aingeru.fernandez@ehu.es}
\usepackage{hyperref} 
\newtheorem{teo}{Theorem}[section]
\newtheorem{prop}{Proposition}[section]
\newtheorem{lem}{Lemma}[section]
\newtheorem{cor}{Corollary}[section]
\theoremstyle{remark}
\newtheorem{rem}{Remark}[section]

\newcommand{\sign}{\operatorname{sign}}
\newcommand{\sumhz}{h\sum_{k=-\infty}^\infty}

\newcommand{\sumhzd}{h^d\sum_{k\in\mathbb{Z}^d}}

\newcommand{\sumhzrec}{h^2\Re\sum_{k=-\infty}^\infty}
\newcommand{\sumhzimc}{h^2\Im\sum_{k=-\infty}^\infty}
\newcommand{\sumzd}{\sum_{k\in\mathbb{Z}^d}}
\newcommand{\Ah}{\mathcal{A}_h}
\newcommand{\Sh}{\mathcal{S}_h}

\begin{document}

\begin{abstract}
In this paper we review the Heisenberg uncertainty principle in a discrete setting and, as in the classical uncertainty principle, we give it a dynamical sense related to the discrete Schr\"{o}dinger equation. We study the convergence of the relation to the classical uncertainty principle, and, as a counterpart, we also obtain another discrete uncertainty relation that does not have an analogous form in the continuous case. Moreover, in the case of the Discrete Fourier Transform, we give a inequality that allows us to relate the minimizer to the Gaussian.
\end{abstract}
\maketitle

\section{Introduction}

The well-known Heisenberg uncertainty principle \cite{c} states that
\begin{equation}\label{eq001}\frac{2}{d}\left(\int_{\mathbb{R}^d}|xf(x)|^2\,dx\right)^{1/2}\left(\int_{\mathbb{R}^d}|\nabla f(x)|\,dx\right)^{1/2}\ge \int_{\mathbb{R}^d}|f(x)|^2\,dx.\end{equation}

Moreover, the minimizing function (that for which \eqref{eq001} is an equality) satisfies, for $\alpha> 0$, $\nabla f(x)+\alpha x f(x)=0\Longrightarrow f(x)=Ce^{-\alpha |x|^2/2}$ (Gaussian). 

Now, if we consider $u(x,t)$ a solution to the Schr\"{o}dinger free equation, there is a dynamic interpretation of the uncertainty principle, which was exploited in \cite{d,e}.

\begin{teo}[Dynamic uncertainty principle]\label{teo11}
Assume $u(x,t)$ is a solution to
\[\left\{\begin{array}{ll}\partial_tu(x,t)=i\Delta u(x,t),& x\in\mathbb{R}^d,t\in\mathbb{R},\\ u(x,0)=u_0(x),&\end{array}\right.\]
where $u_0\in\dot H^1(\mathbb{R}^d)\cap L^2(\mathbb{R}^d,|x|^2\,dx),\ \ \|u_0\|_2^2=2/d.$ For a real function $\phi(x)$ we define
\[
h(t)=\int_{\mathbb{R}^d}\phi(x)|u(x,t)|^2\,dx,\ \ \ a=\int_{\mathbb{R}^d}|x|^2|u_0(x)|^2\,dx<+\infty,\ \ \ b=\int_{\mathbb{R}^d}|\nabla u_0(x)|^2\,dx<+\infty.
\]

Then,
\begin{equation}\label{eq002}
\ddot h(t)=4\int_{\mathbb{R}^d}\nabla u D^2\phi \overline{\nabla u}-\int_{\mathbb{R}^d}\Delta^2\phi|u|^2.\ \ \ \text{(Virial identity)}
\end{equation}

Moreover, if $\phi(x)=|x|^2$,
\begin{equation}\label{eq003}
h(t)=a+4bt^2\ge a+\frac{4t^2}{a},
\end{equation}
and, if these two convex parabolas intersect each other, they are the same parabola and the initial datum is $u_0(x)=Ce^{-\alpha|x|^2/2}$, being then
\[
u(x,t)=\left(\frac{1}{2\alpha it+1}\right)^{d/2}\exp\left(\frac{i\alpha^2t|x|^2}{4\alpha^2t^2+1}-\frac{\alpha|x|^2}{8\alpha^2t^2+2}\right).
\]
\end{teo}

Observe that the normalization condition in the initial datum gives, thanks to the uncertainty principle \eqref{eq001} that $ab\ge 1$. 

In this paper we want to develop this theory in a discrete setting discretizing the momentum and position operators. Since we can relate a sequence to a periodic function via Fourier series, there is a duality between discrete uncertainty principles and periodic uncertainty principles. The relation we study here appears in the literature (see \cite{a,g,b}) in this periodic form. Moreover, in \cite{b} the authors suggested another uncertainty relation. Their aim was to study the angular momentum - angle variables on the sphere, so they related the orbital angular momentum to the azimuthal angle about the $z$ axis. Then, the orbital momentum is written as a differential operator and, for a meaningful uncertainty principle, periodicity is required for the position operator. Hence, the authors suggested the operators $\cos(x)$ and $\sin(x)$ to represent position. Considering this duality via Fourier series, the second case is connected with the discrete version of Heisenberg uncertainty principle that we will study here. In the first case, we will get another relation that does not have a continuous version.

Another version of the Heisenberg uncertainty principle appears in \cite{j,k,g}, but in this case the equality is not attained. However, it is possible to construct a sequence of polynomials $p_k$ of degree $k$ such that the inequality approaches the equality as $k$ tends to infinity. Nevertheless, we will not study this relation here.

As it happens for the Heisenberg uncertainty principle in the continuous case, we will derive Virial identities equivalent to \eqref{eq002} for both relations. Thus we will give them a dynamical interpretation (equivalent to \eqref{eq003}). On the one hand, the dynamics will be given by the discrete Schr\"odinger equation, as it is expected. On the other hand, it will appear an equation that turns out to be an $L^2$-invariant factorization of the one dimensional wave equation.

Since we see an analogy between the continuous and discrete dynamic uncertainty principles, it seems reasonable to have similarities between the solution to the continuous Schr\"odinger equation with initial datum the Gaussian and the solution to the discrete equation, now with initial datum the minimizer of the discrete relation. In the continuous case, it is known that this solution satisfies another equation in the form $(\mathcal{S}+\mathcal{A})\omega=0$, where $\mathcal{S}$ is a symmetric operator and $\mathcal{A}$ is a skew-symmetric operator, so we prove here that in the discrete case the same statement holds.

Apart from this, we consider another discrete setting, the case of finite sequences. The motivation here comes from \cite{h}, where the author gives a relation for the Discrete Fourier Transform, but he suggested that the minimizing sequence of his inequality is not similar to the Gaussian. Here, we will slightly modify this relation in order to see that the minimizer approaches the minimizer of the periodic uncertainty we have mentioned above. Besides, we give two uncertainty principles truncating the operators we will study in Section 2 and imposing periodic or Dirichlet boundary conditions. In these two cases, when the number of nodes tends to infinity we recover the discrete uncertainty principle. However, we will see that we do not have a convex parabola with these versions of the position and momentum operators. This fact is consistent with the periodic Schr\"odinger equation, since there is no convex parabola equivalent to $h(t)$ in Theorem \ref{teo11} in this case.

This paper is organized as follows: In Section 2 we introduce the discrete uncertainty principle we want to study, seeing that the minimizer tends to the Gaussian in the continuous limit. We also discuss the other discrete uncertainty principle related to the $\cos(x)$ operator in the space of periodic functions. In Section 3 we give dynamical interpretations for the uncertainty principles discussed in Section 2. In Section 4 we observe that the continuous and discrete solutions to the Schr\"odinger equation with initial datum  the respective minimizer share some properties. In Section 5 first we give a slight modification of the uncertainty principle stated in \cite{h} that allows us to connect the minimizer to the minimizer of the periodic uncertainty principle of Section 2, and therefore, to the Gaussian. We also truncate the position and momentum operators in Section 2 to consider two cases, the periodic and the Dirichlet case, noticing that we can not repeat the theory we develop in Secion 3.

\section{Uncertainty principle in $H(\mathbb{Z}^d)$}

A useful tool to obtain uncertainty relations is the following (see \cite{f}): Let $\mathcal{S}$ a symmetric operator and $\mathcal{A}$ a skew-symmetric operator in a Hilbert space. Then
\begin{equation}\label{eq004}
|\langle -[\mathcal{S},\mathcal{A}]f,f\rangle| \le 2\|\mathcal{S}f\|\|\mathcal{A}f\|.
\end{equation}

Moreover, the equality is attained when $\alpha \mathcal{S}f+\mathcal{A}f=0$ for $0\ne \alpha\in\mathbb{R}$.

To prove Heisenberg uncertainty principle we set 
\[\mathcal{S}f=xf,\ \mathcal{A}f=\nabla f,\]
so we are going to discretize these operators $\mathcal{S}$ and $\mathcal{A}$. We discretize $\mathbb{R}^d$ with  the same step $h>0$ in all directions, that is, we consider the discretization nodes $x_k=kh$ for $k=(k_1,\dots,k_d)\in\mathbb{Z}^d$, and we are going to work in the space 
\[H(\mathbb{Z}^d)=\{(u_k)_{k\in\mathbb{Z}^d}:  \sumzd|u_k|^2+\sumzd|ku_k|^2<+\infty\}.\]

Now we define our versions of the position and momentum operators
\[
\mathcal{S}_hu_k=khu_k=(k_1h,\dots,k_dh)u_k,\ \ \ \ \ \mathcal{A}_hu_k=\left(\frac{u_{k+e_1}-u_{k-e_1}}{2h},\dots,\frac{u_{k+e_d}-u_{k-e_d}}{2h}\right),
\]
where $e_j=(0,\dots,0,\overbrace{1}^j,0,\dots,0)$, for $j=1,\dots,d$.

It is easy to check that the operators $\Sh$ and $\Ah$ are symmetric and skew-symmetric respectively (with respect to the inner product $\langle u,v\rangle=h^d\sum_{k\in\mathbb{Z}^d} u_k\overline{v_k}$, and we are going to consider this inner product when talking about $\ell^2$ norms). Using \eqref{eq004}, we have the following discrete version of the uncertainty principle: $\forall u\in H(\mathbb{Z}^d),$
\begin{equation}\label{eq005}\begin{aligned}
\left|\sumhzd \sum_{j=1}^d\frac{u_{k+e_j}+u_{k-e_j}}{2}\overline{u_k}\right|&\\\ \le2&\left(\sumhzd|khu_k|^2\right)^{1/2}\left(\sumhzd\sum_{j=1}^d \left|\frac{u_{k+e_j}-u_{k-e_j}}{2h}\right|^2\right)^{1/2}.
\end{aligned}\end{equation}

We can manipulate the left-hand side of \eqref{eq005} to obtain
\begin{equation}\label{eq006}\begin{aligned}
\left|d\sumhzd|u_k|^2-\frac{h^2}{2}\sumhzd\sum_{j=1}^d\left|\frac{u_{k+e_j}-u_{k}}{h}\right|^2\right|\\ \le2\left(\sumhzd|khu_k|^2\right)^{1/2}&\left(\sumhzd\sum_{j=1}^d\left|\frac{u_{k+e_j}-u_{k-e_j}}{2h}\right|^2\right)^{1/2}.
\end{aligned}\end{equation}

In order to take the continuous limit we consider that $u=(u_k)_{k\in\mathbb{Z}^d}$ is the discretization of a function $f(x)\in\mathcal{S}(\mathbb{R}^d)$ (in other words, $u_k=f(x_k)=f(kh)$ for some $f$) and we let $h$ tend to zero. We notice that the second sum in the left-hand side tends to zero when $h$ tends to zero. Indeed, without  the factor $h^2/2$ this sum would tend to $\int_{\mathbb{R}^d}|\nabla f(x)|^2\,dx$, since we have the forward finite difference operator of first order. Therefore, adding the factor $h^2/2$ makes the sum tend to zero. The other sums tend to their respective integrals in the classic Heisenberg uncertainty principle \eqref{eq001}.

Now we will rewrite this inequality in the Fourier space. If we look at $u_k$ as the Fourier coefficient of a $\frac{2\pi}{h}$-periodic function in each variable $f$, we have the following relations between $u$ and $f$,

\begin{equation}\label{eq207}\begin{aligned}
u_k=&\hat{f}(k)=\int_{[-\pi/h,\pi/h]^d}f(\xi)e^{-i\xi\cdot kh}\,d\xi,
\\&f(x)=\left(\frac{h}{2\pi}\right)^d\sum_{k\in\mathbb{Z}^d}u_ke^{ihk\cdot x}.
\end{aligned}\end{equation}  

Considering these relations, we can rewrite the inequality \eqref{eq006}  to have
\begin{equation}\label{eq007}\begin{aligned}
&\left|\int_{[-\pi/h,\pi/h]^d}\sum_{j=1}^d\cos(x_jh)|f(x)|^2\,dx\right|\\&\hspace{1.6 cm} \le2\left(\int_{[-\pi/h,\pi/h]^d}|\nabla f(x)|^2\,dx\right)^{1/2}\left(\int_{[-\pi/h.\pi/h]^d}\sum_{j=1}^d\left|\frac{\sin(x_jh)}{h}f(x)\right|^2\,dx\right)^{1/2}.
\end{aligned}\end{equation}

As it was pointed out in \cite{k}, we have to exclude some cases in \eqref{eq006}. If we want to give an inequality of the type $ab\ge 1$, we need to normalize one quantity that can be zero, so we assume that the function $f$ satisfies
\begin{equation}\label{eq008}
\sum_{j=1}^d\int_{[-\pi/h,\pi/h]^d}\cos(x_jh)|f|^2\,dx\ne 0,
\end{equation}
and, under this assumption we can normalize \eqref{eq007}.

In the sequence space, this condition means that we have to work with sequences such that
\[
\Re\sumzd\sum_{j=1}^d u_k\overline{u_{k+e_j}}\ne 0,
\]
but it is easy to see that the subspace of these sequences is dense in $\ell^2(\mathbb{Z}^d)$. If we are given an $\epsilon>0$ and $0\ne u\in\ell^2(\mathbb{Z}^d)$, then adding $c\epsilon$ in the appropriate coordinate gives us an $\omega$ such that $\|u-\omega\|_2^2=\sumhzd|u_k-\omega_k|^2\le \epsilon$ and $\Re\sumzd\omega_k\overline{\omega_{k+e_j}}\ne 0.$

Once we have this uncertainty relation, we are interested in knowing for which sequence the equality in \eqref{eq006} holds. This sequence which we denote by $\omega^h$ has to satisfy, for $0\ne\alpha\in\mathbb{R}$, the equation $\alpha\Sh\omega^h+\Ah\omega^h=\bf{0}$, where $\bf{0}$ is the sequence whose components are all zero. Then, we have the recurrence relation
\[
\alpha\Sh\omega_k^h+\Ah\omega_k^h=0,\ \ \forall k\in\mathbb{Z}^d\Longleftrightarrow \alpha k_jh\omega_k^h+\frac{\omega_{k+e_j}^h-\omega_{k-e_j}^h}{2h},\ \ \forall k\in\mathbb{Z}^d,\ j=1,\dots,d.
\]
This is the recurrence relation satisfied by a product of modified Bessel functions of the first and second kind. However, we will use the Fourier method to find the minimizing sequence, because the uncertainty principle in the Fourier space is also interesting. If we solve the recurrence looking at $\omega_k^h$ as the Fourier coefficient of a $2\pi/h$-periodic function in each variable $f(x)$, we have
\begin{equation}\label{eq109}
(\alpha\Sh+\Ah)\omega_k^h=0,\forall k\in\mathbb{Z}^d \Longleftrightarrow \alpha \partial_{x_j}f(x)+\frac{\sin(x_jh)}{h}f(x)=0,\ j=1,\dots,d.\end{equation}

Solving the equation, we get
\[
f(x)=C\exp\left(\sum_{j=1}^d\frac{\cos(x_jh)}{\alpha h^2}\right).
\]

We set the constant $C=C_{\alpha,h}$, for example, in order to make the norm in the $L^2[-\pi/h,\pi/h]^d$ space of $f$ equal to 1. We can also set the constant $C$ taking into account the normalization condition \eqref{eq008}, to make that quantity equal to 1.

\begin{rem} In this periodic case, this function is in the appropriate Hilbert space $\forall\alpha\ne 0$, while in the continuous case it is required the extra condition $\alpha>0$. To study convergence to the classic case we will assume that $\alpha>0$.\end{rem}

Once we know who $f(x)$ is, we compute the $k$-th Fourier coefficient of $f(x)$ to get $\omega_k^h$,
\[
\omega_k^h=\int_{[-\pi/h,\pi/h]^d}f(x)e^{-ix\cdot kh}\,dx.
\]

As we have said above, this coefficient is related to the modified Bessel function of the first kind, which has lots of representations, such as
\[
I_m(z)=\frac{1}{\pi}\int_0^\pi e^{z\cos\theta}\cos(m\theta)\,d\theta.
\]
Then, it is easy to check that, under the normalization condition $\omega_{\bf 0}^h=1$,
\[
\omega_{k}^h=\omega_{k_1,\dots,k_d}^h=\prod_{i=1}^d\frac{I_{k_i}\left(\frac{1}{\alpha h^2}\right)}{I_{0}\left(\frac{1}{\alpha h^2}\right)}.
\]

The modified Bessel function of the second kind $K_{k_j}(1/\alpha h^2)$ also satisfies this recurrence relation (if we multiply it by the factor $(-1)^{k_j}$), but this sequence is not in $\ell^2(\mathbb{Z}^d)$, so it makes no sense to consider this sequence, and this is the reason why we only get the modified Bessel function of the first kind using the Fourier method.

We are going to take $\alpha=1$ for simplicity, and we will see the convergence of the minimizer to $e^{-|x|^2/2}$. The same proof is valid for each value of $\alpha>0$. The way to approach the Gaussian is to take $h$ and $k=(k_1,\dots,k_d)$ in a proper way such that $kh$ approach to a given point $x=(x_1,\dots,x_d)\in\mathbb{R}^d$. We can do this defining, for $j\in\mathbb{N}$,
\[
h_j=\max_{i\in\{1,\dots,d\}}\{|x_i|\}/j,\ \ k_m^{(j)}=\left\{\begin{array}{ll}\lceil x_m/h_j\rceil,&\text{if }x_m\ge0,\\\lfloor x_m/h_j\rfloor,&\text{if }x_m\le0.\end{array}\right.
\]

Here $\lceil z\rceil$ and $\lfloor z\rfloor$ stand for the \textit{ceiling} and \textit{floor} functions. Notice that $\lfloor -z^2\rfloor =-\lceil z^2\rceil.$ Now we define the function
\[
f_j(x)=\left\{\begin{array}{cc}\displaystyle\prod_{m=1}^d\frac{I_{k_m^{(j)}}(1/h_j^2)}{I_{0}(1/h_j^2)}=\omega_{k_1^{(j)},\dots,k_d^{(j)}}^{h_j},& \text{if }x\ne{\bf0},\\ 1&\text{if }x={\bf0}.\end{array}\right.
\]
and this is the function that will give us the connection between the two minimizers, when $j$ goes to infinity. It is quiet easy to see that $f_j$ is an even function in each variable. We have the following result.
\begin{teo}
Given $\epsilon>0$, There is $j_0$ such that if $j\ge j_0$ then we have,
\[
\sup_{x\in\mathbb{R}^d}\left|f_j(x)-e^{-|x|^2/2}\right|<\epsilon.
\]
\end{teo}

\textit{Proof.}  Since the function $f_j$ is even, we can assume that $x_m\ge0,\forall m=1\ \dots,d$. On the other hand, if $x={\bf0}$ then there is nothing to prove, so we can assume that at least one variable is not 0. Furthermore, the symmetry of the problem tells us that if we write
\[
\mathbb{R}^d_{+}=\bigcup_{m=1}^d\{x\in\mathbb{R}^d{+}:x_m=\max\{x_1,\dots,x_d\}\},
\]
then the proof in each region will be the same. Hence, we just need to see the convergence in the set where $x_1$ is the maximum of the components of $x$. Thanks to this consideration we have
\[
h_j=x_1/j,\ \ k_1^{(j)}=j,\ \ k_m^{(j)}=\left\lceil\frac{x_mj}{x_1}\right\rceil,
\] 
Once this has been settled, we can start proving the convergence. To begin with, we split the difference of the minimizer and the gaussian to get (notice that $e^{-z}\le 1$ and  $I_n(z)<I_0(z)$ when $n\in\mathbb{N},\ z>0$).
\[
\left|f_j(x)-e^{-|x|^2/2}\right|\le\left|\frac{I_j(j^2/x_1^2)}{I_0(j^2/x_1^2)}-e^{-x_1^2/2}\right|+\sum_{m=2}^d\left|\frac{I_{\lceil x_mj/x_1\rceil}(j^2/x_1^2)}{I_0(j^2/x_1^2)}-e^{-x_m^2/2}\right|.
\]

We are going to treat each piece of the last inequality separately, and the proof is the same for each part.  Here we show the proof of the part $m=2$, that is, we have to deal with
\[
\left|\frac{I_{\lceil x_2j/x_1\rceil}(j^2/x_1^2)}{I_0(j^2/x_1^2)}-e^{-x_2^2/2}\right|.
\]
For the first part, we are going to use two known asymptotic expressions for the modified Bessel function. One is a uniform asymptotic expression (see \cite[p.~377]{olver}) for $I_\nu(\nu z)$ when $\nu\rightarrow \infty$ valid in $0<z<\infty$ and the other one is an asymptotic expression for $I_0(t)$ when $t\rightarrow\infty$ (see \cite[p.~269]{olver},\cite[p.~203]{l}). More precisely,
\[
\left|\frac{\sqrt{2\pi\nu}(1+z^2)^{1/4}I_{\nu}(\nu z)}{e^{\nu\xi_z}}-1\right|\le \frac{3}{5\nu},
\]
where $\xi_z=\sqrt{1+z^2}+\log\left(\frac{z}{1+\sqrt{1+z^2}}\right)$. On the other hand,
\begin{equation}\label{eqi0}
\left|\frac{\sqrt{2\pi t}I_0(t)}{e^{t}}-1\right|\le \frac{1}{t}.
\end{equation}

Then, we take $M$ and $a>0$ to be chosen later, and in the sequel $C$ will denote a constant which only depends on $M$ and $a$. If $a\le x_2\le x_1\le M$, from the last two estimates we get   
\[
\left|\frac{\sqrt{2\pi\lceil x_2j/x_1\rceil}(1+j^4/x_1^4\lceil x_2j/x_1\rceil^2)^{1/4}I_{\lceil x_2j/x_1\rceil}(j^2/ x_1^2)}{e^{\lceil x_2j/x_1\rceil\xi_j}}-1\right|\le \frac{3}{5\lceil x_2 j/x_1\rceil}\le\frac{3x_1}{5x_2j}\le\frac{C}{j},
\]
\[
\left|\frac{\sqrt{2\pi j}I_0(j^2/ x_1^2)}{x_1e^{j^2/ x_1^2}}-1\right|\le \frac{x_1^2}{j^2}\le\frac{C}{j^2}.
\]

Here $\xi_j=\sqrt{1+j^4/x_1^4\lceil x_2j/x_1\rceil^2}+\log\frac{j^2}{x_1^2\lceil x_2j/x_1\rceil+\sqrt{ x_1^4\lceil x_2j/x_1\rceil^2+j^4}}.$ It is easy to check that the convergence of the minimizer to the Gaussian will be given by the study of,
\[
\left|\frac{e^{\lceil x_2j/x_1\rceil\xi_j-j^2/x_1^2+x_2^2/2}}{(1+x_1^4\lceil x_2j/x_1\rceil^2/j^4)^{1/4}}-1\right|.
\]

Now, using that $(1+z)^{-1/4}=1+O(z)$ we have
\[
z=\frac{\lceil x_2j/x_1\rceil^2x_1^4}{j^4}\le\frac{(x_2j/x_1+1)^2x_1^4}{j^4}\le\frac{C}{j^2},
\]
so
\[
(1+\lceil x_2j/x_1\rceil^2x_1^4/j^4)^{-1/4}=1+O(1/j^2).
\]

Notice that the big O notation gives us a constant which only depends on $M$ and $a$. On the other hand, when $0<z<1$, $\log(1-z)=-z+O\left(\frac{z^2}{(1-z)^2}\right)$ and we have $\log\frac{j^2}{ x_1^2\lceil x_2j/x_1\rceil+\sqrt{x_1^4\lceil x_2j/x_1\rceil^2+j^4}}=\log\left(1-\frac{ x_1^2\lceil x_2j/x_1\rceil+\sqrt{x_1^4\lceil x_2j/x_1\rceil^2+j^4}-j^2}{x_1^2\lceil x_2j/x_1\rceil+\sqrt{x_1^4\lceil x_2j/x_1\rceil^2+j^4}}\right)$, so since the logarithm is multiplied by $\lceil x_2j/x_1\rceil$, to study the error term we look at the quantity
\[
\frac{z\lceil x_2j/x_1\rceil^{1/2}}{1-z}=\frac{2j^2x_1^2\lceil x_2j/x_1\rceil^{3/2}}{j^4-j^2x_1^2\lceil x_2j/x_1\rceil+j^2\sqrt{x_1^4\lceil x_2j/x_1\rceil^2+j^4}}\le \frac{2x_1^{1/2}x_2^{3/2}}{j^{1/2}}+\frac{C}{j^{3/2}}\le\frac{C}{j^{1/2}}.
\]

Thus, manipulating the quotient inside the logarithm we get
\[\begin{aligned}
\lceil x_2j/x_1\rceil\xi_j-j^2/x_1^2+x_2^2/2&=x_2^2-\frac{2j^2x_1^2\lceil x_2j/x_1\rceil^{2}}{j^4+j^2x_1^2\lceil x_2j/x_1\rceil+j^2\sqrt{x_1^4\lceil x_2j/x_1\rceil^2+j^4}}\\&-\frac{x_2^2}{2}+\frac{j^2}{x_1^2}\left(\sqrt{x_1^4\lceil x_2j/x_1\rceil^2/j^4+1}-1\right)+O(1/j).
\end{aligned}\]

If we consider the first line, using now the Taylor expansion of $\sqrt{1+z}$ and the fact that $\lceil x_2j/x_1\rceil=x_2j/x_1+O(1)$, $\lceil x_2j/x_1\rceil^2=x_2^2j^2/x_1^2+O(j)$ we observe that it is bounded by $\frac{C}{j}$. In other words,
\[
x_2^2-\frac{2j^2x_1^2\lceil x_2j/x_1\rceil^{2}}{j^4+j^2x_1^2\lceil x_2j/x_1\rceil+j^2\sqrt{x_1^4\lceil x_2j/x_1\rceil^2+j^4}}=O(1/j).
\]

For the second part, we use that $\sqrt{1+z}=1+z/2+O(z^2)$, and it is easy to check that then
\[
-\frac{x_2^2}{2}+\frac{j^2}{x_1^2}\left(\sqrt{x_1^4\lceil x_2j/x_1\rceil^2/j^4+1}-1\right)=O(1/j),
\]
and finally we have
\[
\frac{e^{\lceil x_2j/x_1\rceil\xi_j-j^2/x_1^2+x_2^2/2}}{(1+\lceil x_2j/x_1\rceil^2x_1^4/j^4)^{1/4}}=(1+O(1/j^2))e^{O(1/j)}=1+O(1/j),
\]
or, in other words
\[
\left|\frac{e^{\lceil x_2j/x_1\rceil\xi_j-j^2/x_1^2+x_2^2/2}}{(1+\lceil x_2j/x_1\rceil^2x_1^4/j^4)^{1/4}}-1\right|\le \frac{C}{j}.
\]

Now we can use this estimate to go back to the quantity we want to control and conclude that
\[
\left|\frac{I_{\lceil x_2j/x_1\rceil}(j^2/ x_1^2)}{I_0(j^2/ x_1^2)}-e^{-x_2^2/2}\right|\le \frac{C}{j}.
\]

Hence, if we assume that $a\le x_m\le x_1\le M,\ \forall m=2,\dots,d$, there is $j_1$ such that if $j\ge j_1$,
\[
\left|f_j(x)-e^{-|x|^2/2}\right|\le \epsilon.
\]

Now we take $a$ and $M$ in order to have
\[
e^{-M^2/2}\le\epsilon,\ \ 1-e^{-a^2/2}\le\epsilon,\ 4a^2 e^{2a^2}\le\epsilon,
\]
and we will see that then, in the other regions of the set where $x_1$ is the maximum, we have that the difference between the minimizer and the Gaussian is less than $\epsilon$. First, for $x_1\le M$, we study the case when some variables are less than $a$. For all the variables $x_m$ bigger than $a$ we can repeat the proof of the first part, so we only have to deal with those variables that are less than $a$. We will assume here without loss of generality that $x_2$ is less than $a$. In this region we are going to use another estimate which can be deduced from an asymptotic expansion for $I_\nu(z)$ given in \cite[p.~269]{olver}. More precisely, 
\[
\left|\frac{\sqrt{2\pi z}I_\nu(z)}{e^z}-1\right|\le \frac{\pi(4\nu^2-1)}{8z}e^{\pi(4\nu^2-1)/8z}+e^{-2z}\left(1+\frac{4\nu^2-1}{4z}e^{(4\nu^2-1)/8z}\right).
\]

Therefore, when $\nu=\lceil x_2j/x_1\rceil$ and $z=j^2/x_1^2$, there is $j_2$ such that if $j\ge j_2$, for all $x_2<a,\ x_1\le M$ and $x_2<x_1$ we have
\[
\frac{4\nu^2-1}{z}\le \frac{4x_1^2}{j^2}\left(\frac{x_2j}{x_1}+1\right)^2\le 4a^2+\frac{C}{j}\le 5a^2,\  e^{-2z}\le e^{-2j^2/M^2}\le a^2e^{2a^2},
\]
hence we have
\[
\left|\frac{\sqrt{2\pi j}I_{\lceil x_2j/x_1\rceil}(j^2/x_1^2)}{x_1e^{j^2/x_1^2}}-1\right|\le \frac{5\pi a^2}{8}e^{5\pi a^2/8}+a^2 e^{2a^2}(1+5a^2e^{a^2/8}/4)\le 4a^2e^{2a^2}\le\epsilon.
\]

We can use this and \eqref{eqi0} to check that for $j$ big enough and independent of $x_1$ and $x_2$,
\[
\left|\frac{I_{\lceil x_2j/x_1\rceil}(j^2/ x_1^2)}{I_0(j^2/ x_1^2)}-e^{-x_2^2/2}\right|\le\left|\frac{I_{\lceil x_2j/x_1\rceil}(j^2/ x_1^2)}{I_0(j^2/ x_1^2)}-1\right|+1-e^{-x_2^2/2}\le 3\epsilon.
\]

We repeat this argument for all the variables that are less than $a$ in order to get the desired result. Thus, we have that there is $j_0$ such that if $j\ge j_0$, the difference between the minimizer and the Gaussian is less than $\epsilon$ for all $x\in\mathbb{R}^d_+$ such that $x_1\le M$ is the maximum variable. Now we have to check the case when $x_1> M$. In this case, we can bound $f_j$ using the following property of the modified Bessel functions:

\begin{lem}
$\frac{I_\nu(t)}{I_0(t)}$ is an increasing function for $t>0$.
\end{lem}

\textit{Proof.} Differentiating we have
\[
\frac{I_\nu'(t)}{I_0(t)}-\frac{I_\nu(t)I_0'(t)}{I_0^2(t)}>0\Longleftrightarrow Y_\nu(t)=\frac{I_\nu'(t)}{I_\nu(t)}>Y_0(t),
\]
and in \cite{amos} the author proved that $Y_{\nu+1}(t)>Y_\nu(t)$ for $t>0$ and $\nu\ge0$. \qed

Hence, by the Lemma and the first part of the proof we have (notice again that $I_n(z)<I_0(z)$)
\[
f_j(x)\le \frac{I_j(j^2/x_1^2)}{I_0(j^2/x_1^2)}\le \frac{I_j(j^2/M^2)}{I_0(j^2/M^2)}\le e^{-M^2/2}+\epsilon\le2\epsilon,
\]
while, on the other hand $e^{-|x|^2/2}\le e^{-M^2/2}\le \epsilon,$ so, finally we get that, if $x_1>M$,
\[
\left|f_j(x)-e^{-|x|^2/2}\right|\le \max\{f_j(x),e^{-|x|^2/2}\}\le 2\epsilon.
\]

Thus, we have covered all the posibilities when the maximum variable is $x_1$. Since we can repeat this process for all the variables, the desired result holds. \qed

\begin{rem}
Using the uniform convergence and a proper bound for $f_j$, we can also see that the convergece holds in $L^1(\mathbb{R}^d)$, and therefore, by interpolation, we have convergece in $L^p(\mathbb{R}^d)$ for all $p\in[1,\infty]$.
\end{rem}

Since we have this duality between this uncertainty principle and the uncertainty principle for periodic functions, we can also see the convergence of the periodic minimizer to the Gaussian by letting $h$ tend to zero, and then the period of the periodic function goes to infinity. In this direction, the convergence is proved in \cite{g}, where the authors do not use our parameter $h$ and let $\alpha$ tend to zero. Nevertheless, we can introduce $h$ in their proof and use the same argument to have the convergence to the gaussian when $h$ tends to zero.

This uncertainty principle \eqref{eq006} is not new, as we have pointed out in the introduction. In \cite{b}, the authors used the inequality we have used here (in one dimension and in the Fourier space). Since they also considered the position operator given by $\cos(x)$, they presented another uncertainty relation in their paper. In order to get convergence, we put the uncertainty relation in the following way
\begin{equation}\label{eq010}
2\left(\int_{-\pi/h}^{\pi/h}|f'|^2\right)^{1/2}\left(\int_{-\pi/h}^{\pi/h}\left|\cos(xh)f\right|^2\right)^{1/2}\ge\left|h\int_{-\pi/h}^{\pi/h}\sin(xh)|f|^2\right|.
\end{equation}

Although in this case, when $h$ tends to zero this relation does not converge to any uncertainty relation since the right-hand side goes to zero, the study of this relation in the discrete setting can be interesting.

The discrete operators which give this inequality are
\[
\widetilde{\mathcal{S}_h}u_k=khu_k,\ \ \ \widetilde{\mathcal{A}_h}u_k=i\frac{u_{k+1}+u_{k-1}}{2}.
\]

Notice that we multiply by $i$ so that $\widetilde{\mathcal{A}_h}$ is skew-symmetric and now $d=1$ so $k$ is a \textit{number}, not a \textit{tuple}. The continuous versions of these operators are 
\[
\widetilde{\mathcal{S}}f=xf,\ \ \ \widetilde{\mathcal{A}}f=i f,
\]
and we see here that in the continuous setting we do not have an analogous uncertainty relation since these operators commute and we would have the relation
\[
2\left(\int_{\mathbb{R}}|xf|^2\right)^{1/2}\left(\int_{\mathbb{R}}|f|^2\right)^{1/2}\ge 0.
\]

If we calculate the minimizing sequence, it corresponds in Fourier  with the periodic function $g(x)=C e^{\sin(xh)/\alpha h}$ and in the sequence space with $\omega_k^h=C_{\alpha,h}i^{-k}I_k\left(1/\alpha h\right)$, so we have almost the same sequence (forgetting the $h$) we had before. In Fourier, it is quite easy to check that the minimizing function goes to zero when $h$ goes to zero. It makes sense to have the same mimizing sequence with a factor $i^k$ since we can go (assume for a moment $h=1$ and $d=1$) from \eqref{eq007} to \eqref{eq010} by doing the change of variables $y=x-\pi/2$ which gives the factor $i^k$ in the sequence space.

The difference between the two cases is that in the first case the minimizing's center was fixed, but now it depends on $h$, as we can see in Figure \ref{im1}. Moreover, in the first case the value at the center of the Gaussian goes to a constant, but in the second case it goes to zero.

\begin{figure}[\h]\centering
\begin{subfigure}[b]{0.4\textwidth}
\includegraphics[scale=0.5]{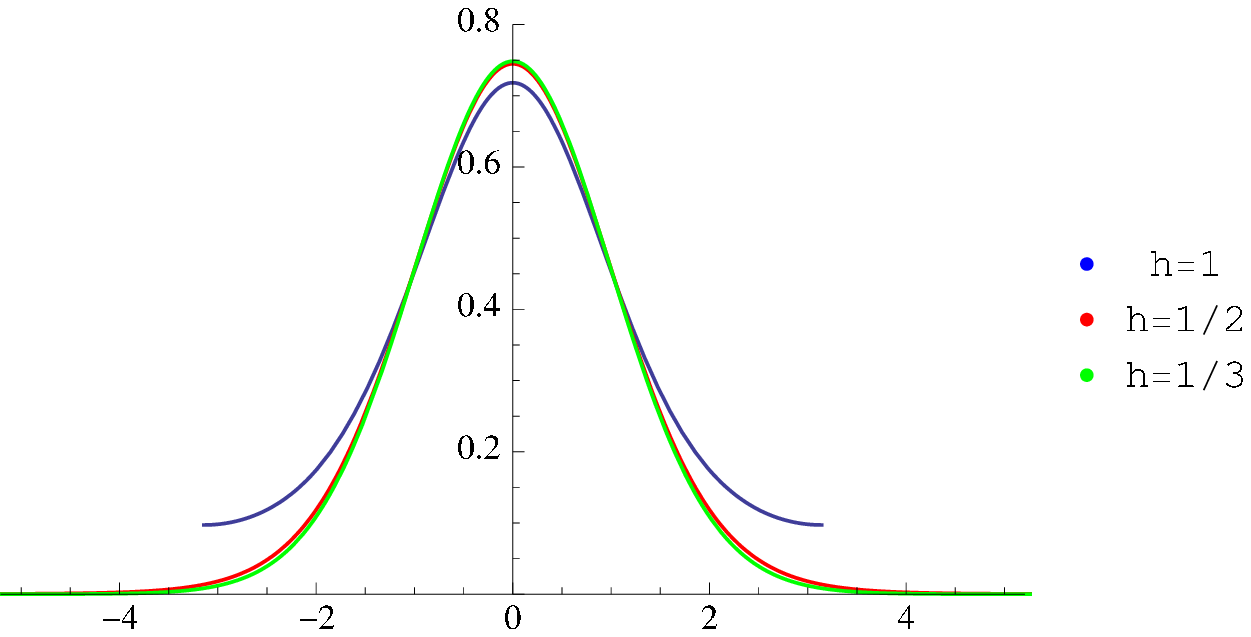}
\end{subfigure}\qquad\qquad\begin{subfigure}[b]{0.4\textwidth}
\includegraphics[scale=0.4]{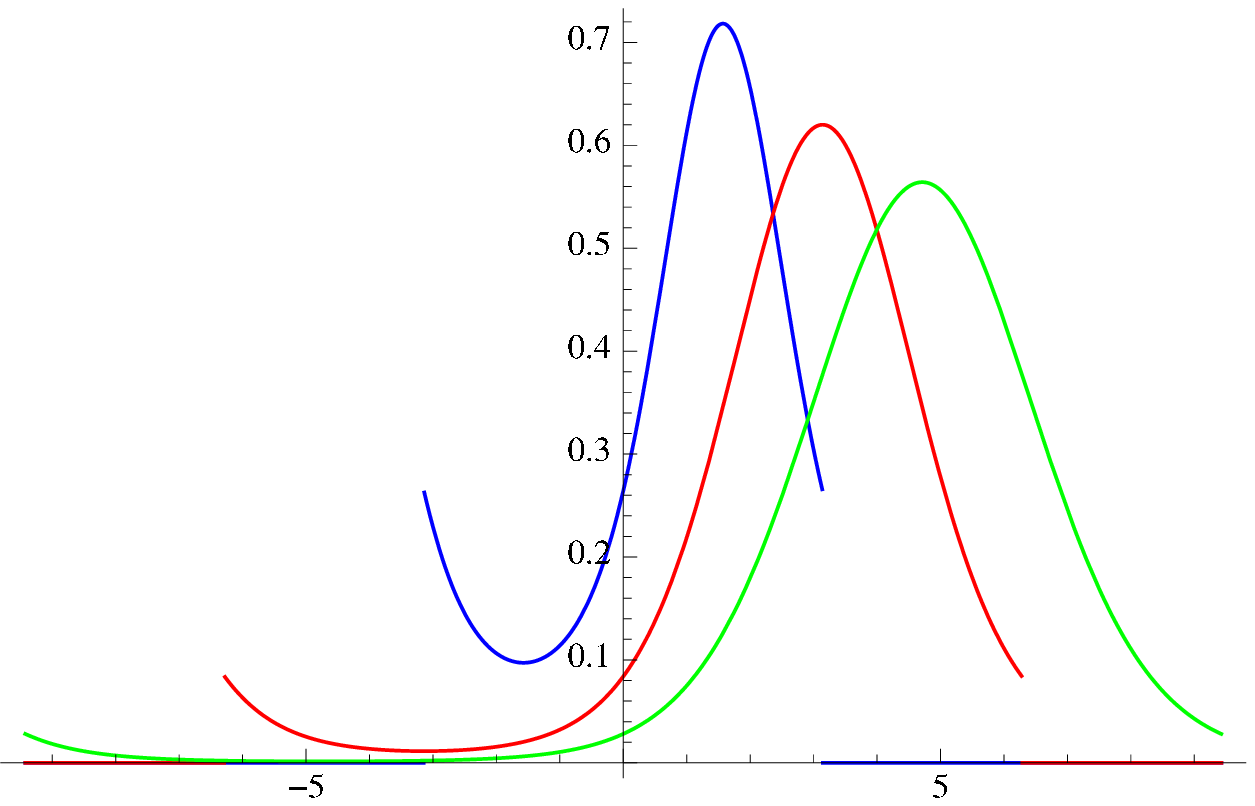}
\end{subfigure}
\caption{Minimizing function $f(x)$ (left) and $g(x)$ (right) \label{im1}}
\end{figure}

In this case the discrete uncertainty relation is

\begin{equation}\label{eq011}\begin{aligned}
\left|h^2\sum_{k=-\infty}^\infty \frac{u_{k+1}-u_{k-1}}{2}\overline{u_k}\right|&\\ \le2&\left(\sumhz|khu_k|^2\right)^{1/2}\left(\sumhz \left|\frac{u_{k+1}+u_{k-1}}{2}\right|^2\right)^{1/2}.
\end{aligned}\end{equation}

\section{Virial identity}

In this section we are going to give a \textbf{discrete Virial identity} equivalent to \eqref{eq002}, which relates evolution equations to the inequalities \eqref{eq006} and \eqref{eq011}. Then we will use it to obtain a dynamic uncertainty principle equivalent to \eqref{eq003}.

First of all we define the discrete Laplacian as the composition of the backward and the forward difference operators, that is,
\[
\Delta_du_k=\sum_{j=1}^d\frac{u_{k+e_j}-2u_k+u_{k-e_j}}{h^2}=\sum_{j=1}^d\partial_j^+\partial_j^-u_k,
\]
where
\[
\partial_j^+ u_k=\frac{u_{k+e_j}-u_k}{h}\Rightarrow \partial_j^-u_k=\frac{u_{k}-u_{k-e_j}}{h}.
\]

Notice that $(\partial_j^+)^*=-\partial_j^-$.

We have the following result, equivalent to \eqref{eq002},\eqref{eq003}:
\begin{teo}[Dynamic discrete uncertainty principle]\label{teo31}
Assume $u=(u_k(t))_k$ is a solution to the discrete Schrödinger equation
\[
\left\{\begin{array}{lr}\partial_tu_k=i\Delta_du_k,& k\in\mathbb{Z}^d,\ t\in\mathbb{R},\\u_k(0)=u_k^0,\end{array}\right.
\]
where $u^0=(u_k^0)_k\in H(\mathbb{Z}^d)$ such that
\begin{equation}\label{eq012}
\sumhzd\sum_{j=1}^du_k^0\overline{u_{k+e_j}^0}=2.
\end{equation}

For a real $\phi=(\phi_k)_{k\in\mathbb{Z}^d}$ we define
\[
F(t)=\sumhzd \phi_k|u_k(t)|^2,\ \ a=\sumhzd|khu_k^0|^2<+\infty,\ \ b=\sumhzd\sum_{j=1}^d\left|\frac{u_{k+e_j}^0-u_{k-e_j}^0}{2h}\right|^2<+\infty.
\]

Then, if $\phi_k$ is of the form $\phi_{k_1}+\dots+\phi_{k_d}$,
\begin{equation}\label{eq013}
\ddot{F}(t)=4\sumhzd D^2\phi_k\Ah u_k\overline{\Ah u_k}-\sumhzd\Delta_d^2\phi_k|u_k|^2,
\end{equation}
where
\[
D^2\phi_k=\left(\begin{array}{ccc}\partial_1^+\partial_1^-\phi_k&&0\\
&\ddots&\\
0&&\partial_d^+\partial_d^-\phi_k\end{array}\right).
\]

Moreover, if $\phi_k=(k_1^2+\cdots+k_d^2)h^2$,
\begin{equation}\label{eq014}
F(t)=a+4bt^2\ge a+\frac{4t^2}{a},
\end{equation}
and, if these two convex parabolas intersect each other, they are the same parabola and the initial datum is $u_k^0=\omega_k^h=C_{\alpha,h}I_k\left(1/\alpha h^2\right)$, being the solution
\begin{equation}\label{eq015}
u_k(t)=e^{it\Delta_d}\omega_k^h=\omega_k^h(t)= e^{-2dit/h^2}C_{\alpha,h}\prod_{j=1}^dI_{k_j}\left(\frac{1+2\alpha it}{\alpha h^2}\right).
\end{equation}
\end{teo}

\begin{rem} Observe that the Hessian is a diagonal matrix since  $\partial_j^{\pm}\partial_{l}^{\pm}\phi_k=0$ for $j\ne l$. If we do not make this extra assumption on $\phi$, we get some extra terms in the expression of $\ddot{F}(t).$\end{rem}

\begin{rem} Notice that now $C_{\alpha,h}$ changes because of the normalization condition \eqref{eq012}.\end{rem}

\begin{proof} For convenience, we are going to use the notation that we have used in the continuous case. That is, we say that $u$ satisfies $\partial_tu=i\sum_{j=1}^d\partial_j^+\partial_j^- u$ and we denote $F(t)=\int \phi(x)|u(x,t)|^2\,dx$.

We will need a discrete Leibniz rule, and there are many ways to write this discrete rule. For example,
\[
\partial_j^+(\phi_ku_k)=\frac{\phi_{k+e_j}u_{k+e_j}-\phi_ku_k}{h}=\frac{\phi_{k+e_j}-\phi_k}{h}u_k+\frac{u_{k+e_j}-u_k}{h}\phi_k+h\frac{u_{k+e_j}-u_k}{h}\frac{\phi_{k+e_j}-\phi_k}{h},
\]
which we denote $\partial_j^+(\phi u)=u\partial_j^+\phi+\phi\partial_j^+ u+h\partial_j^+ u\partial_j^+\phi$. In the same way, we have another Leibniz rule for $\partial_j^-$, $\partial_j^-(\phi u)=u\partial_j^-\phi+\phi\partial_j^- u-h\partial_j^- u\partial_j^-\phi$. Moreover, we have
\begin{equation}\label{eq016}\begin{aligned}
\partial_j^+ u+\partial_j^-u&=2\partial_j^s u\text{ (symmetric difference operator),}
\\ \partial_j^+ u-\partial_j^- u&=h\partial_j^+\partial_j^- u.
\end{aligned}\end{equation}

Then, taking a time derivative we have, formally
\[
\dot{F}(t)=2\Re\int\phi u\overline{\partial_t u}=2\Im\sum_{j=1}^d\int \phi u\overline{\partial_j^+\partial_j^- u}=-2\Im\sum_{j=1}^d \int\partial_j^+(\phi u)\overline{\partial_j^+ u}=-2\Im\sum_{j=1}^d\int \partial_j^+\phi u\overline{\partial_j^+ u}.\]

Taking another derivative and using the assumption on $\phi_k$, the Leibniz rule and \eqref{eq016},
\[\begin{aligned}
\ddot{F}(t)&=4\Re\sum_{j=1}^d\int\partial_j^+\partial_j^-\phi\partial_j^s u\overline{\partial_j^+ u}-\sum_{j=1}^d\int(\partial_j^+\partial_j^-)^2\phi|u|^2+h\sum_{j=1}^d\int\partial_j^+\partial_j^-\partial_j^+\phi \partial_j^+\partial_j^- u\overline{u}\\&\ \ \ +2h\sum_{j=1}^d\int \partial_j^+\partial_j^-\partial_j^+\phi|\partial_j^+ u|^2+h\sum_{j=1}^d\int(\partial_j^+\partial_j^-)^2\phi\partial_j^-u\overline{u}.
\end{aligned}\]

On the other hand, following the same procedure, but interchanging the role of $\partial_j^+$ and $\partial_j^-$, we get
\[\begin{aligned}
\ddot{F}(t)&=4\Re\sum_{j=1}^d\int\partial_j^+\partial_j^-\phi\partial_j^s u\overline{\partial_j^- u}-\sum_{j=1}^d\int(\partial_j^+\partial_j^-)^2\phi|u|^2-h\sum_{j=1}^d\int\partial_j^-\partial_j^+\partial_j^-\phi \partial_j^+\partial_j^- u\overline{u}\\&\ \ \ -2h\sum_{j=1}^d\int \partial_j^-\partial_j^+\partial_j^-\phi|\partial_j^- u|^2-h\sum_{j=1}^d\int(\partial_j^+\partial_j^-)^2\phi\partial_j^+u\overline{u}.
\end{aligned}\]

The sum of these two formulae and \eqref{eq016} gives 
\[\begin{aligned}
2\ddot{F}(t)&=8\int D^2\phi\Ah u\overline{\Ah u}-2\int\Delta_d^2\phi|u|^2\\&\ \ \ +2h\sum_{j=1}^d\int \partial_j^+\partial_j^-\partial_j^+\phi|\partial_j^+ u|^2-2h\sum_{j=1}^d\int \partial_j^-\partial_j^+\partial_j^-\phi|\partial_j^- u|^2.
\end{aligned}\]

Finally, we notice that, for $j=1,\dots,d$,
\[
2h\int \partial_j^+\partial_j^-\partial_j^+\phi|\partial_j^+ u|^2-2h\int \partial_j^-\partial_j^+\partial_j^-\phi|\partial_j^- u|^2=0.\]

Hence,
\[
\ddot{F}(t)=4\sumhzd D^2\phi_k\Ah u_k\overline{\Ah u_k}-\sumhzd\Delta_d^2\phi_k|u_k|^2.
\]

Now, as in the continuous case, we take $\phi_k=|x_k|^2=|kh|^2=h^2(k_1^2+\dots+k_d^2)$, the discretization of $|x|^2$, and we get the two terms of the right-hand side of \eqref{eq006}. Indeed,
\[
D^2\phi_k=2I_{d},\ \ \Delta_d^2\phi_k=0,
\]
where $I_d$ is the identity matrix of order $d\times d$. Then,
\[\begin{aligned}
F(t)&=\sumhzd|khu_k(t)|^2,\\ \ \ddot{F}(t)&=8\sumhzd\sum_{j=1}^d\left|\frac{u_{k+e_j}-u_{k-e_j}}{2h}\right|^2.
\end{aligned}\]

Moreover, $\dddot{F}(t)=0$.

We can see this fact looking at our equation in the Fourier space. If we consider $u_k(t)=\hat{f}(k,t)$, the equation $\partial_t u_k=i\Delta_d u_k$ is equivalent to
\[
\partial_t f(x,t)=i\sum_{j=1}^d\frac{2\cos(x_jh)-2}{h^2}f(x,t)=-4i\sum_{j=1}^d\frac{\sin^2(x_jh/2)}{h^2}f(x,t),\] 
whose solution is
\[
f(x,t)=\exp\left(-4i\sum_{j=1}^d\frac{\sin^2(x_jh/2)t}{h^2}\right)f(x,0).
\]

Then it is quite obvious that the $L^2[-\frac{\pi}{h},\frac{\pi}{h}]^d$ norm of $f$ is preserved, and so it is the $\ell^2$ norm of $u$. Since $\Big((u_{k+e_j}-u_{k-e_j})/2h\Big)_k$ satisfies the same equation, $\ddot{F}(t)=\ddot{F}(0)$. In the same way, we can see that the normalization condition \eqref{eq012} is also preserved with the time. To make these calculations rigorous we also use the equation in the Fourier space. Then, thanks to the expression of the solution, it is quite easy to check that
\begin{equation}\label{eq017}
F(t)=\|\nabla f(t)\|_{L^2[-\pi/h,\pi/h]}\le C(t\|f(0)\|_{L^2[-\pi/h,\pi/h]}+\|\nabla f(0)\|_{L^2[-\pi/h,\pi/h]})<+\infty,
\end{equation}
since $u^0$ is in $H(\mathbb{Z}^d)$. In fact, we can refine these estimate of $F(t)$ to give it in terms of $a$ and $b$, but, since we are working on discrete spaces, $b$ is controlled by $\|u^0\|_2$.

These facts allow us to write $F(t)$ as a convex parabola, and, we can assume without loss of generality $\dot F(0)=0$ (if not, we make a translation in time). Then $F(t)=a+4bt^2$. Furthermore, by \eqref{eq006} and \eqref{eq012}, the coefficients of this parabola satisfy the inequality
\[\sqrt{ab}\ge 1,\]
so \eqref{eq014} holds.

As in the continuous case, if the equality holds, then we know that for some $\alpha$, $(u_k^0)_k=(\omega_k^h)_k$ is the minimizing sequence. If we solve the equation
\[
\left\{\begin{array}{lr}\partial_tu_k(t)=i\Delta_du_k(t),&k\in\mathbb{Z}^d,\ t\in\mathbb{R},\\ u_k(0)=\omega_k^h= C_{\alpha,h}\prod_{j=1}^dI_{k_j}\left(1/\alpha h^2\right),\end{array}\right.
\]
we get, by properties of the modified Bessel functions (see \cite{l})
\[\begin{aligned}
e^{it\Delta_d}\omega_k^h=u_k(t)&= e^{-2dit/h^2}C_{\alpha,h}\prod_{j=1}^d\sum_{m_j\in\mathbb{Z}}I_{m_j}\left(\frac{1}{\alpha h^2}\right)I_{k_j-m_j}\left(\frac{2it}{h^2}\right)\\&= e^{-2dit/h^2}C_{\alpha,h}\prod_{j=1}^dI_{k_j}\left(\frac{1+2\alpha it}{\alpha h^2}\right),
\end{aligned}\]
hence, the solution has the same form as the initial datum, both are products of modified Bessel functions of the first kind.\end{proof}

\begin{rem} For $\gamma\in\mathbb{R}$, if we consider the equation
\[
\partial_tu_k(t)=i\sum_{j=1}^d\frac{u_{k+e_j}+\gamma u_k+u_{k-e_j}}{h^2},
\]
and we repeat all the calculations, we get to the same result. Note that multiplying $u_k$ by an appropriate exponential term we can reduce this equation to
\[
\partial_tv_k(t)=i\sum_{j=1}^d\frac{v_{k+e_j}+v_{k-e_j}}{h^2},
\]
so dealing with this equation is enough to see the general case with $\gamma$, and, in particular, the discrete Schr\"{o}dinger equation $(\gamma=-2)$.\end{rem}

Since in Fourier we also have the other periodic uncertainty principle, one can think that this new discrete uncertainty relation \eqref{eq011}, although it is not a discrete version of Heisenberg uncertainty principle, satisfies another Virial identity, but the natural choice for the equation (the composition of the ``backward summation operator'' and the ``forward summation operator'' fails, fact that we have pointed out in the previous remark, since it would be the case $\gamma=2$.

The question then is:
Is there any equation (we restrict ourselves to the one dimensional case) $\partial_tu_k=\mathcal{T}u_k$ $\ell^2(\mathbb{Z})$-invariant such that
\[
F(t)=\sumhz|khu_k(t)|^2\Rightarrow \ddot F(t)=C\sumhz\left|\frac{u_{k+1}(t)+u_{k-1}(t)}{2}\right|^2,\ \ \dddot F(t)=0\ ?
\]

We present here two equations that answer the question in an affirmative way. Since these two equations are not very different, we present a general result and later we will talk about the equations in detail.

\begin{teo}\label{teo32}
Assume $u=(u_k(t))_k$ and $v=(v_k(t))_k$ satisfy 
\[\left\{\begin{array}{ll}
\partial_tu_k=i\dfrac{v_{k+1}-v_{k-1}}{2h},& \partial_tv_k=-i\dfrac{u_{k+1}-u_{k-1}}{2h},\\u_k(0)=u_k^0,& v_k(0)=v_k^0,\end{array}\right.
\]
 where $u^0=(u_k^0)_k,\ v=(v_k^0)_k\in H(\mathbb{Z})$ and
 \[h^2\sum_{k=-\infty}^\infty \frac{u_{k+1}^0-u_{k-1}^0}{2}\overline{u_k^0}=2.\]
 
 We define
 \[
 F(t)=\sumhz k^2h^2|u_k(t)|^2,\ \ a=\sumhz k^2h^2|u_k^0|^2<+\infty,\ \ b=\sumhz\left|\frac{u_{k+1}^0+u_{k-1}^0}{2}\right|^2<+\infty.
 \]
 
Then, if $\ \forall k\in\mathbb{Z},\ \forall t,\ |u_k(t)|^2=|v_k(t)|^2$ and  $\Re(u_{k+1}(t)\overline{u_{k-1}(t)})=\Re(v_{k+1}(t)\overline{v_{k-1}(t)})$,
\[
\ddot F(t)=2\sumhz\left|\frac{u_{k+1}(t)+u_{k-1}(t)}{2}\right|^2,\ \ \dddot F(t)=0,
\]
and the system is $\ell^2(\mathbb{Z})$-invariant. Moreover,
\[
F(t)=a+bt^2\ge a+\frac{t^2}{a},
\]
and, if these two parabolas intersect each other, they are the same parabola and the initial datum $u^0$ is the minimizer of \eqref{eq011}.
\end{teo}

\begin{proof} To begin with, we will prove the $\ell^2(\mathbb{Z})$-invariance. If we differentiate the equality given by the hypothesis $\|u(t)\|_2=\|v(t)\|_2$, then we have
\[
\Re\sum_{k=-\infty}^\infty i(v_{k+1}-v_{k-1})\overline{u_k}=-\Re\sum_{k=-\infty}^\infty i(u_{k+1}-u_{k-1})\overline{v_k}.
\]

Then adding this sums, dividing by 2 and using that $\Im(z)=-\Im(\overline{z})$, we have
\[
\partial_t\|u(t)\|_2=\frac{1}{2}\Im\sum_{k=-\infty}^\infty(\overline{v_k}u_{k-1}-\overline{v_k}u_{k+1}+u_{k+1}\overline{v_k}-u_{k-1}\overline{v_k})=0.
\]

Using the same procedure, we observe that
\[\partial_t\Re\sum_{k=-\infty}^\infty u_{k+1}\overline{u_{k-1}}=\partial_t\Re\sum_{k=-\infty}^\infty v_{k+1}\overline{v_{k-1}}=0,\]
 which will be useful later.

Now,
\[
\dot F(t)=\sumhzrec k^2i(v_{k+1}-v_{k-1})\overline{u_k}=-\sumhzimc k^2(v_{k+1}\overline{u_k}-v_{k-1}\overline{u_k}).
\]

Differentiating again we have
\[\begin{aligned}
\ddot F(t)&=-\frac{h}{2}\Re\sum_{k=-\infty}^\infty  k^2(-u_{k+2}\overline{u_k}+2|u_k|^2-|v_{k+1}|^2-|v_{k-1}|^2+2v_{k+1}\overline{v_{k-1}}-u_{k-2}\overline{u_k})\\&=\frac{h}{2}\Re\sum_{k=-\infty}^\infty\Bigl((k-1)^2(u_{k+1}\overline{u_{k-1}}+|v_k|^2)-2k^2(|u_k|^2+v_{k+1}\overline{v_{k-1}})\bigr.\\&\hspace{2.2cm} \bigl.+(k+1)^2(|v_k|^2+u_{k+1}\overline{u_{k-1}}) \Bigr)\\&=h\Re\sum_{k=-\infty}^\infty\left(|u_k|^2+u_{k+1}\overline{u_{k-1}}\right)=2\sumhz\left|\frac{u_{k+1}+u_{k-1}}{2}\right|^2.
\end{aligned}\]

Moreover, using the previous calculations it is now obvious that $\dddot F(t)=0$. Again, this formal calculations are rigorous if we look at the system in the Fourier space, having a similar estimate to \eqref{eq017}. Now, assuming without loss of generality $\dot F(0)=0$, and since by \eqref{eq011}
\[
\sqrt{a b}\ge 1,
\]
we have $F(t)=a+bt^2\ge a+t^2/a$,
and if these two parabolas intersect each other, we have the equality in \eqref{eq011} and then $u^0$ has to be the minimizing sequence.
\end{proof}

Now we are going to solve the system using the Fourier method. The system that we want to solve is
\[
\left\{\begin{array}{ll}\partial_tu_k=i\dfrac{v_{k+1}-v_{k-1}}{2h}&u_k(0)=u_k^0,\\ \partial_tv_k=-i\dfrac{u_{k+1}-u_{k-1}}{2h}&v_k(0)=v_k^0.\end{array}\right.
\]

We consider $f_0(x),g_0(x)$ $2\pi/h-$periodic functions such that 
\[
u_k^0=\hat{f_0}(k),\ v_k^0=\hat{g_0}(k)\text{ and }u_k(t)=(f(t))\hat\,(k),\ v_k(t)=(g(t))\hat\, (k),
\]
so the system in the Fourier space is
\[
\left\{\begin{array}{ll}\partial_tf=\dfrac{\sin(xh)}{h}g&f(x,0)=f_0(x),\\ \partial_tg=-\dfrac{\sin(xh)}{h}f&g(x,0)=g_0(x).\end{array}\right.
\]

We have a system of two ODEs, whose solution is
\[\begin{aligned}
f(x,t)&=\frac{f_0(x)-ig_0(x)}{2}e^{i\sin(xh)t/h}+\frac{f_0(x)+ig_0(x)}{2}e^{-i\sin(xh)t/h},\\
g(x,t)&=\frac{g_0(x)+if_0(x)}{2}e^{i\sin(xh)t/h}+\frac{g_0(x)-if_0(x)}{2}e^{-i\sin(xh)t/h}.
\end{aligned}\]

Finally, we recover from these expressions the value of $u_k(t)$ and $v_k(t)$.
\[
u_k(t)=\int_{-\pi/h}^{\pi/h}f(x,t)e^{-ixkh}\,dx,\ \ v_k(t)=\int_{-\pi/h}^{\pi/h}g(x,t)e^{-ixkh}\,dx.
\]

We can prove the $\ell^2(\mathbb{Z})-$invariance in the Fourier space too, proving that the functions $f$ and $g$ are $L^2[-\pi/h,\pi/h]-$invariant. We only have to use that $\|f(t)\|_2=\|g(t)\|_2\ \forall t$, which is true thanks to the hypothesis on $u_k$ and $v_k$.

The two equations we want to mention here are the cases when we set, on the one hand $v_k=\overline{u_k}$ and, on the other hand, $v_k=(-1)^{k+1}u_k$. It is easy to check that these two options satisfy the hypothesis of the Theorem \ref{teo32}.

\textit{\underline{First case: $v_k=\overline{u_k}$.}}

In this case we can state the Virial principle as follows:

\begin{cor}
Assume $u=(u_k(t))_k$ satisfies 
\[
\partial_tu_k=i\frac{\overline{u_{k+1}}-\overline{u_{k-1}}}{2h},\ \ \ u_k(0)=u_k^0,
\]
 where $u^0=(u_k^0)_k\in H(\mathbb{Z})$ and
 \[h^2\sum_{k=-\infty}^\infty \frac{u_{k+1}^0-u_{k-1}^0}{2}\overline{u_k^0}=2.\]
 
 We define
 \[
 F(t)=\sumhz k^2h^2|u_k(t)|^2,\ \ a=\sumhz k^2h^2|u_k^0|^2<+\infty,\ \ b=\sumhz\left|\frac{u_{k+1}^0+u_{k-1}^0}{2}\right|^2<+\infty.
 \]
 
Then,
\[
\ddot F(t)=2\sumhz\left|\frac{u_{k+1}^0+u_{k-1}^0}{2}\right|^2,
\]
and the equation is $\ell^2(\mathbb{Z})$-invariant. Moreover,
\[
F(t)=a+bt^2\ge a+\frac{t^2}{a},
\]
and, if these two parabolas intersect each other, they are the same parabola and the initial datum $u^0$ is the minimizer of \eqref{eq011}.
\end{cor}

In this case, since $g_0(x)=\overline{f_0}(-x)$, the solution to the equation is
\[
u_k(t)=\int_{-\pi/h}^{\pi/h}f(x,t)e^{-ixkh}\,dx,
\]
where
\[
f(x,t)=\frac{f_0(x)-i\overline{f_0}(-x)}{2}e^{i\sin(xh)t/h}+\frac{f_0(x)+i\overline{f_0}(-x)}{2}e^{-i\sin(xh)t/h}.
\]

In this case, the equation is a discrete version of the equation
\[
\partial_t f=i\overline{\partial_x f}.
\]

If we take another time derivative, we get the wave equation $\partial_t^2f=\partial_x^2 f$, so this equation gives an $L^2$-invariant factorization of the one dimensional wave equation. Using \textit{d'Alembert's formula} we can see that the solution to this equation is
\[
f(x,t)=\frac12\Big(f_0(x+t)+f_0(x-t)\Big)+\frac{i}{2}\Big(\overline{f_0}(x+t)-\overline{f_0}(x-t)\Big).
\]

If we take another time derivative in our discrete equation, as it can be expected we get a discrete version of the wave equation
\[
\partial_t^2u_k=\frac{u_{k+2}-2u_k+u_{k-2}}{4h^2}.
\]

In the continuous setting, the analogous of this corollary is:
\begin{prop}
Assume $f$ satisfies
\[\partial_t f=i\overline{\partial_xf},\ \ f(x,0)=f_0(x),\]
and let $F(t)=\int_\mathbb{R}x^2|f(x,t)|^2\,dx,$ then
\[
\ddot{F}(t)=2\int_\mathbb{R}|f(x,t)|^2=\ddot{F}(0).
\]
\end{prop}

\textit{\underline{Second case: $v_k=(-1)^{k+1}u_k$.}}

In this case we can state the Virial principle as follows:

\begin{cor}
Assume $u=(u_k(t))_k$ satisfies 
\[
\partial_tu_k=i\frac{(-1)^ku_{k+1}+(-1)^{k-1}u_{k-1}}{2h},\ \ \ u_k(0)=u_k^0,
\]
 where $u^0=(u_k^0)_k\in H(\mathbb{Z})$ and
 \[h^2\sum_{k=-\infty}^\infty \frac{u_{k+1}^0-u_{k-1}^0}{2}\overline{u_k^0}=2.\]
 
 We define
 \[
 F(t)=\sumhz k^2h^2|u_k(t)|^2,\ \ a=\sumhz k^2h^2|u_k^0|^2<+\infty,\ \ b=\sumhz\left|\frac{u_{k+1}^0+u_{k-1}^0}{2}\right|^2<+\infty.
 \]
 
Then,
\[
\ddot F(t)=2\sumhz\left|\frac{u_{k+1}^0+u_{k-1}^0}{2}\right|^2,
\]
and the equation is $\ell^2(\mathbb{Z})$-invariant. Moreover,
\[
F(t)=a+bt^2\ge a+\frac{t^2}{a},
\]
and, if these two parabolas intersect each other, they are the same parabola and the initial datum $u^0$ is the minimizer of \eqref{eq011}.
\end{cor}

In this case, since $g_0(x)=-f_0\left(x+\pi/h\right)$, the solution to the equation is
\[
u_k(t)=\int_{-\pi/h}^{\pi/h}f(x,t)e^{-ixkh}\,dx,
\]
where
\[
f(x,t)=\frac{f_0(x)+if_0\left(x+\pi/h\right)}{2}e^{i\sin(xh)t/h}+\frac{f_0(x)-if_0\left(x+\pi/h\right)}{2}e^{-i\sin(xh)t/h}.
\]

\section{Properties of $e^{it\Delta_d}u^0$}

Now we will see that the function (see \eqref{eq015}) $\omega_k^h(t)=e^{it\Delta_d}\omega_k^h$, where $(\omega_k^h)_k$ is the minimizing function to \eqref{eq006}, and the solution to the Schrödinger equation $g(x,t)$ with initial datum $g_0(x)=e^{-\alpha |x|^2/2}$ have got similar properties.

We recall that $g_0(x)$ satisfies the equation $\alpha x g_0(x)+\nabla g_0(x)=0$, which is a sum of a symmetric and a skew-symmetric operator. It is easy to see that then $g(x,t)$ satisfies 
\[
g=(\alpha \mathcal{S}+\mathcal{A})g=0,\text{ where }\mathcal{A}u=\nabla u,\ \ \mathcal{S}u=x u+2it\nabla u.
\]

Hence, the solution $g(x,t)$ satisfies another equation with a symmetric and a skew-symmetric operator. Moreover, if we denote $\Lambda(t)f=xf+2it\nabla f$, we can see that
\[
\Lambda(t)e^{it\Delta}u_0(x)=e^{it\Delta}\Lambda(0)u_0(x)\Longrightarrow \Lambda(t)=e^{it\Delta}\Lambda(0)e^{-it\Delta}=e^{it\Delta}xe^{-it\Delta},
\] 
where $e^{it\Delta}u_0(x)$ stands for the solution to the problem
\[
\left\{\begin{array}{ll}\partial_tu(x,t)=i\Delta u(x,t),&x\in\mathbb{R}^d,\ t\in\mathbb{R},\\u(x,0)=u_0(x).\end{array}\right.
\]

We have the following result:

\begin{teo}\label{teo41}
Let $\omega^h(t)=(\omega_k^h(t))_k$ be given by \eqref{eq015}. Then
\[(\mathcal{A}_h+\alpha\Lambda_d(t))\omega_k^h(t)=0,\]
where $\mathcal{A}_h$ is skew-symmetric and $\Lambda_d(t)$ is symmetric and given by
\[
\mathcal{A}_hu_k(t)=\left(\frac{u_{k+e_j}(t)-u_{k-e_j}(t)}{2h}\right)_j, \ \ \ \Lambda_d(t)u_k(t)=khu_k(t)+2it\mathcal{A}_hu_k(t).
\]

Moreover,
\[
\Lambda_d(t)e^{it\Delta_d}u_k^0=e^{it\Delta_d}\Lambda_d(0)u_k^0\Longrightarrow \Lambda_d(t)=e^{it\Delta_d}\Lambda_d(0)e^{-it\Delta_d}=e^{it\Delta_d}khe^{-it\Delta_d}.
\]
\end{teo}

\begin{proof} As the skew-symmetric operator in both equations is the same in the continuous case, $\mathcal{A}u=~\nabla u$, we will compute $\mathcal{A}_h\omega_k^h(t)$ and we will see that we get the symmetric operator from there. Using the recurrence of the function $I_{k_j}(z)$ we have, for $j=1,\dots,d$,
\[\begin{aligned}
\frac{\omega_{k+e_j}^h(t)-\omega_{k-e_j}^h(t)}{2h}&=\frac{e^{-{2dit}/{h^2}}C_{\alpha,h}}{2h}\left(I_{k_j+1}\left(\frac{1+2\alpha it}{\alpha h^2}\right)-I_{k_j-1}\left(\frac{1+2\alpha it}{\alpha h^2}\right)\right)\prod_{l\ne j}I_{k_l}\left(\frac{1+2\alpha it}{\alpha h^2}\right)\\&=-\frac{\alpha k_jh e^{-{2dit}/{h^2}}C_{\alpha,h}}{1+2\alpha it}\prod_{l=1}^dI_{k_l}\left(\frac{1+2\alpha it}{\alpha h^2}\right)=-\frac{\alpha k_jh}{1+2\alpha it}\omega_k^h(t).
\end{aligned}\]

Hence, $\omega_k^h(t)$ satisfies the equation $(\mathcal{A}_h+\alpha\Lambda_d)\omega_k^h(t)=0$.

 Furthermore, using again the recurrence of $I_k(z)$ we have
\[\begin{aligned}
&\Lambda_d(t)e^{it\Delta_d}u_k^0=khe^{-2dit/h^2}\sum_{m\in\mathbb{Z}^d}u_{m}^0\prod_{l=1}^dI_{k_l-m_l}\left(\frac{2it}{h^2}\right)\\&\ \ \ +\left(\frac{it}{h}e^{-2dit/h^2}\sum_{m\in\mathbb{Z}^d}u_{m}^0\left(I_{k_j-m_j+1}\left(\frac{2it}{h^2}\right)-I_{k_j-m_j-1}\left(\frac{2it}{h^2}\right)\right)\prod_{l\ne j}I_{k_l-m_l}\left(\frac{2it}{h^2}\right)\right)_j\\&=\left(e^{-2dit/h^2}\sum_{m\in\mathbb{Z}^d}u_{m}^0\left(k_jh-\frac{it}{h}\frac{2(k_j-m_j)h^2}{2it}\right)\prod_{l=1}^dI_{k_l-m_l}\left(\frac{2it}{h^2}\right)\right)_j\\&=\left(e^{-2dit/h^2}\sum_{m\in\mathbb{Z}^d}m_jhu_m^0\prod_{l=1}^dI_{k_l-m_l}\left(\frac{2it}{h^2}\right)\right)_j=e^{it\Delta_d}\Lambda_d(0)u_0^k.\end{aligned}\]
\end{proof}

\section{Uncertainty principles for finite sequences}

In this section we are going to see some uncertainty relations for finite sequences in one dimension $u=(u_k)_{k=-N}^{k=N}$. The motivation comes from \cite{h}, where the author gives an uncertainty relation for the DFT considering discrete versions of the position and momentum operators, but, using his words, the minimizer does not \textit{``bear much of a connection with the natural of the Gaussian in this context''.} Here, we introduce a slight modification of his operators in order to relate the new minimizer to the Gaussian. The main difference between this approach and the one in \cite{h} is that here we introduce a new parameter which allows us to recover the Gaussian in a limiting process which consists in two steps. First we recover the minimizing function of the periodic uncertainty principle \eqref{eq007}, and then, as we have seen in Section 2 we approach the Gaussian when the period of the minimizing function tends to infinity. Moreover, we give two uncertainty relations truncating the operators we have studied in Section 2 and assuming periodic and Dirichlet conditions.

\subsection{The case of the Discrete Fourier Transform}

The operators we propose here are
\begin{equation}\label{018}
\mathcal{S}_h=\left[\begin{array}{ccc}
q_{-N}&&0\\
&\ddots&\\
0&&q_N
\end{array}\right],\ \ \
\mathcal{A}_h=\frac{1}{2h}\left[\begin{array}{ccccc}
0&1&0&\cdots&-1\\
-1&0&1&\cdots&0\\
0&-1&0&\ddots&0\\
 & &\ddots&\ddots&\\
 1&0&\cdots&-1&0\end{array}\right],
\end{equation}
where
\[
q_k=\frac{(2N+1)h}{2\pi}\sin\left(\frac{2\pi k }{2N+1}\right),\ \ k=-N,\dots, N.
\]

\begin{rem} In  \cite{h}, the author considered the coefficients (in this case for sequences $(u_k)_{k=0}^{k=N}$ and $h=1/2$)
\[
\tilde{q}_k=\sin\left(\frac{2\pi k}{N}\right).
\]

With this choice of $\tilde{q}_k$, the uncertainty principle in \cite{h} has a nice representation for $\|\mathcal{A}_{h}u\|_2$ in terms of the DFT, but, as we have said above, there is no relation between the minimizer  and the Gaussian.\end{rem}

Then if we consider the DFT of a sequence
\[
\hat{u}_k=\frac{1}{\sqrt{2N+1}}\sum_{j=-N}^Nu_je^{-2\pi i k j/(2N+1)},\ \ k=-N,\dots,N,
\]
the uncertainty principle can be written as
\begin{equation}\label{eq119}
2\left(h\sum_{k=-N}^Nq_k^2|u_k|^2\right)^{1/2}\left(h\sum_{k=-N}^N\frac{\sin^2\left(\frac{2\pi k}{2N+1}\right)}{h^2}|\hat{u}_k|^2\right)^{1/2}\ge \left|\langle-[\mathcal{S}_h,\mathcal{A}_h]u,u\rangle\right|,
\end{equation}
or
\begin{equation}\label{eq1192}
2\left(h\sum_{k=-N}^Nq_k^2|u_k|^2\right)^{1/2}\left(h\sum_{k=-N}^N\left(\frac{2\pi}{(2N+1)h^2}\right)^2q_k^2|\hat{u}_k|^2\right)^{1/2}\ge \left|\langle-[\mathcal{S}_h,\mathcal{A}_h]u,u\rangle\right|.
\end{equation}

As we know, the minimizer $\omega^h=(\omega_{k,N}^h)_{k=-N,N}$ satisfies the relation $(\mathcal{S}_h+\alpha\mathcal{A}_h)\omega^h=0$, for $\alpha\ne 0$. Here we will assume that $\alpha=1$ and the initial condition $\omega_{0,N}^h=1$.  Now we want to relate this minimizer to the classical Gaussian, but what we are going to see is that this minimizer converges to the minimizer of the periodic uncertainty principle sated in Section 2.

From \eqref{eq119} and \eqref{018} we know that the minimizing sequence satisfies the system, for $k=-N,\dots,N$
\[
q_k\omega_{k,N}^h+\frac{\omega_{k+1,N}^h-\omega_{k-1,N}^h}{2h}=0,\ \ k=-N,\dots,N
\]
with the conditions $\omega_{0,N}^h=1,\ \omega_{N+1,N}^h=\omega_{-N,N}^h$ and $\omega_{-N-1,N}^h=\omega_{N,N}^h$. Now we define the function $f_j^L(x)$ as
\[
f_j^L(x)=\left\{\begin{array}{cc}\omega_{\sign(x)j,\lceil jL/|x|\rceil}^{|x|/j},&\text{if }0\ne |x|\le L,\\1,&\text{if }0=x< L.\end{array}\right.
\]

The equation that solves the minimizing sequence is a discrete version of the equation, for $x\in[-L,L]$
\[
\frac{L}{\pi}\sin\left(\frac{\pi x}{L}\right)\omega(x)+\omega'(x)=0.
\]

Therefore, we should have that the continuous limit of the sequence should be the  minimizing function of the periodic uncertainty principle shown in Section 2  \eqref{eq007} and \eqref{eq109}, now with the initial condition $\omega(0)=1$, and the role of $h$ played by the quantity $\frac{\pi}{L}$. Hence, as we have shown in Section 2, if we let $L$ tend to $\infty$, then we recover the Gaussian. In Figure \ref{im}, we can see how the minimizing sequence approaches the minimizing function of the periodic uncertainty principle.

\begin{figure}[\h]\centering
\includegraphics{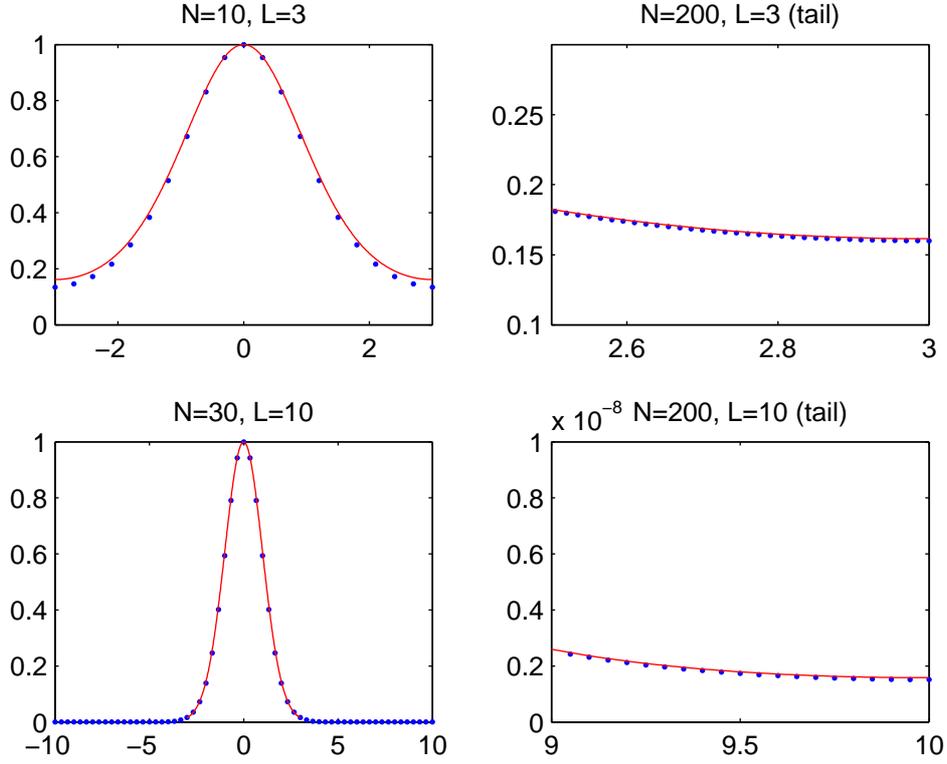}
\caption{Graphic representation of the minimizing sequence and the minimizer of the periodic uncertainty principle in two cases. We see here that when $L$ is large the minimizing sequence approaches the Gaussian. We also see that the convergence in the tails is slower than in the center of the interval.}\label{im}
\end{figure}

In order to see the convergence of the minimizer we slightly change $q_k$ to
\[
q_k=\frac{Nh}{\pi}\sin\left(\frac{\pi k }{N}\right),\ \ k=-N,\dots, N.
\]
and the general case follows directly from this case. The result is the following:
\begin{teo}
Given $x$ and $L>0$ such that $x\in[-L,L]$,
\[
\lim_{j\rightarrow\infty}f_j^L(x)=e^{L^2(\cos(\pi x/L)-1)/\pi^2}.
\]
\end{teo}

\textit{Proof.} To begin with, we point out that, since $q_0=0$, we have, by induction $\omega_{k,N}^h=\omega_{-k,N}^h$, so from now on we will only have in mind $x$ positive. On the other hand,  if $x=0$, then we do not have nothing to prove. Moreover, this symmetry in the minimizing sequence allows us to construct the solution to the system by an iterative process starting from $\omega_{N,N}^h$ to $\omega_{1,N}^h$. We have then that
\[\omega_{k,N}^h=\frac{1}{[2hq_k,\dots,2hq_{N-1},1+2hq_N]}\frac{1}{[2hq_{k-1},\dots,1+2hq_N]}\cdots\frac{1}{[2hq_1,\dots,1+2hq_N]},\]
where
\[[a_0,a_1,\cdots,a_n]=a_0+\frac{1}{a_1+\frac{1}{\ddots+\frac{1}{a_n}}}.\]

To deal with this product of continued fractions, we use Theorem 149 in \cite{i},  which states that the continued fraction $[a_0,a_1,\dots,a_r]$ is a rational number $\frac{p_r}{q_r}$, where $p_r$ and $q_r$ are given by the recurrence
\[\begin{aligned}
&p_0=a_0,\ p_1=a_1a_0+1,\ p_n=a_np_{n-1}+p_{n-2}\ (2\le n\le r),\\
&q_0=1,\ q_1=a_1,\ q_n=a_nq_{n-1}+q_{n-2}\ (2\le n\le r).
\end{aligned}\]

Hence, $f_j^L(x)=\omega_{j,\lceil jL/x\rceil}^{x/j}=\frac{s_j}{t_j}$, where
\[\begin{aligned}
s_j&=\left(\begin{array}{cc}1+2\frac{x}{j}q_{\lceil jL/x\rceil}&1\end{array}\right)\prod_{m=\lceil jL/x\rceil-1}^{j+2}\left(\begin{array}{cc}2\frac{x}{j}q_m&1\\1&0\end{array}\right)\left(\begin{array}{c}2\frac{x}{j}q_{j+1}\\1\end{array}\right),\\
t_j&=\left(\begin{array}{cc}1+2\frac{x}{j}q_{\lceil jL/x\rceil}&1\end{array}\right)\prod_{m=\lceil jL/x\rceil-1}^3\left(\begin{array}{cc}2\frac{x}{j}q_m&1\\1&0\end{array}\right)\left(\begin{array}{c}4\frac{x^2}{j^2}q_1q_2+1\\2\frac{x}{j}q_1\end{array}\right).
\end{aligned}\]

\begin{rem}
The notation $\prod_{m=\lceil jL/x\rceil-1}^{j+2}\left(\begin{array}{cc}2\frac{x}{j}q_m&1\\1&0\end{array}\right)$ represents that the first matrix is the one with index $m=\lceil jL/x\rceil-1$, the following matrix is the one with index $m=\lceil jL/x\rceil-2$, and so on.
\end{rem}

 We will assume here that $j,\lceil\frac{jL}{x}\rceil\equiv 0$, (mod 4) and the other cases follow a similar argument. Moreover, we will study separately the behaviour of the numerator and the denominator.

In the case of the numerator, we can write $s_j$ in the following way:
\[
s_j=\sum_{u=0}^{\lceil Lj/x\rceil-j}a_u,\ \text{where} 
\]
\[
a_0=1,\ \ a_u=\left(\frac{2x}{j}\right)^u\sum_{l_1=j/2+1}^{\left\lceil\frac{\lceil L j/x\rceil-u}{2}\right\rceil}\ \sum_{l_2=l_1}^{\left\lceil\frac{\lceil L j/x\rceil-u}{2}\right\rceil}\dots\sum_{l_u=l_{u-1}}^{\left\lceil\frac{\lceil L j/x\rceil-u}{2}\right\rceil}q_{2l_1-1}q_{2l_2}\dots q_{2l_u+u-2},
\]
for $1\le u\le \lceil \frac{L j}{x}\rceil-j$.

When $u$ is fixed, $a_u$ converges to an integral expression when $j$ tends to infinity. To clarify this, we consider the case $u=1$, that is, the sum
\[
\frac{2\pi}{\lceil Lj/x\rceil}\sum_{l=j/2+1}^{\lceil\frac{Lj}{x}\rceil/2}\left(\left\lceil\frac{Lj}{x}\right\rceil\frac{x}{j\pi}\right)^2\sin\left(\frac{\pi(2l-1)}{\lceil Lj/x\rceil}\right).
\]

This sum represents a partition of step $\frac{1}{\lceil Lj/x\rceil}$ of the interval $\left[\frac{j/2+1}{\lceil Lj/x\rceil},\frac12\right]$. Moreover, we have that $\frac{L}{\pi}\le \lceil\frac{Lj}{x}\rceil\frac{x}{j\pi}\le \frac{L}{\pi}+\frac{x}{j\pi}$. This and the fact that when $j$ tends to infinity the interval tends to $\left[\frac{x}{2L},\frac12\right]$ imply that we can bound from below and from above the limit by the same quantity, so we can conclude that
\[
\lim_{j\rightarrow \infty}a_1=\frac{2L^2}{\pi}\int_{x/2L}^{1/2}\sin(2\pi z)\,dz.
\]

For the general case, $a_u$ will converge to an iterated integral by the same reasons. More precisely,
\[\begin{aligned}
\lim_{j\rightarrow\infty}a_u&=\left(\frac{2L^2}{\pi}\right)^u\int_{x/2L}^{1/2}\int_{x_1}^{1/2}\dots\int_{x_{u-1}}^{1/2}\sin(2\pi x_1)\dots\sin(2\pi x_u)\,dx_u\dots dx_1\\&=\left(\frac{2L^2}{\pi}\right)^u\left(\int_{x/2L}^{1/2}\sin(2\pi z)\,dz\right)^u\frac{1}{u!}=\left(\frac{L^2}{\pi^2}(1+\cos(\pi x/L))\right)^u\frac{1}{u!}.
\end{aligned}\]

Now we are going to see that we can interchange the limit with the sum, using Weierstrass criterion. For that, we are going to bound all the sine functions by 1 and get bounds that are independet of $u$. Bounding the sine functions we get
\[\begin{aligned}
\sum_{u=0}^{\lceil Lj/x\rceil-j}a_u&\le\sum_{u=0}^{\lceil Lj/x\rceil-j}\left(\frac{2x}{j}\right)^u\binom{\frac{\lceil Lj/x\rceil-j}{2}+\left[\frac{u}{2}\right]}{u}\\&\le\left(\sum_{u=0}^{\frac{\lceil Lj/x\rceil-j}{2}}+\sum_{j=\frac{\lceil Lj/x\rceil-j}{2}}^{\lceil Lj/x\rceil-j}\right)\binom{\frac{\lceil Lj/x\rceil-j}{2}+\left[\frac{u}{2}\right]}{u}=I+II.
\end{aligned}\]

To begin with, we can make $II$ as small as we want when $j$ is big enough. Indeed, these binomial coefficients form two decreasing sequences, one is generated by the case $u$ even and the other one by the case $u$ odd. Assume that $u=2m$ is even,
\[
\binom{\frac{\lceil Lj/x\rceil-j}{2}+m}{2m}\ge\binom{\frac{\lceil Lj/x\rceil-j}{2}+m+1}{2m+2}\Leftrightarrow5m^2+7m+2\ge \left(\frac{\lceil Lj/x\rceil-j}{2}\right)^2+\frac{\lceil Lj/x\rceil-j}{2},
\] 
which is true because $2m\ge\frac{\lceil Lj/x\rceil-j}{2}$. On the other hand, if $u=2m+1$ is odd, then
\[
\binom{\frac{\lceil Lj/x\rceil-j}{2}+m}{2m+1}\ge\binom{\frac{\lceil Lj/x\rceil-j}{2}+m+1}{2m+3}\Leftrightarrow5m^2+12m+7\ge \left(\frac{\lceil Lj/x\rceil-j}{2}\right)^2,
\]
which is true as well. Moreover, it is quite obvious to check that (recall that\newline $j,\lceil Lj/x\rceil\equiv 0$ (mod 4), so $\frac{\lceil Lj/x\rceil-j}{2}$ is even)
\[
\binom{\frac{\lceil Lj/x\rceil-j}{2}+\frac{\lceil Lj/x\rceil-j}{4}}{\frac{\lceil Lj/x\rceil-j}{2}}\ge\binom{\frac{\lceil Lj/x\rceil-j}{2}+\frac{\lceil Lj/x\rceil-j}{4}}{\frac{\lceil Lj/x\rceil-j}{2}+1}.
\]

Therefore, we have that $a_u\le \left(\frac{2x}{j}\right)^{\frac{\lceil Lj/x\rceil-j}{2}}\binom{3\frac{\lceil Lj/x\rceil-j}{4}}{\frac{\lceil Lj/x\rceil-j}{2}},\ \forall u\ge\frac{\lceil Lj/x\rceil-j}{2}$. Observe that we can improve this estimate since in this way we are decreasing the power of $\frac{x}{j}$ in each $a_u$ to the power of $\frac{x}{j}$ in $a_{\frac{\lceil Lj/x\rceil-j}{2}}$, but this is enough to prove the convergence. The last bound allows us to say that
\[
\sum_{u=\frac{\lceil Lj/x\rceil-j}{2}}^{\lceil Lj/x\rceil-j}a_u\le\left(\frac{2x}{j}\right)^{\frac{\lceil Lj/x\rceil-j}{2}}\binom{3\frac{\lceil Lj/x\rceil-j}{4}}{\frac{\lceil Lj/x\rceil-j}{2}}\frac{\lceil Lj/x\rceil-j}{2}.
\]

The next step consists in proving that this number tends to zero when $j$ tends to infinity. To prove that, we are going to use the Stirling's approximation
\[
\sqrt{2\pi}n^{n+1/2}e^{-n}\le n!\le n^{n+1/2}e^{-n+1},
\]
we have that, after some manipulations
\[\begin{aligned}
\left(\frac{2x}{j}\right)^{\frac{\lceil Lj/x\rceil-j}{2}}&\binom{3\frac{\lceil Lj/x\rceil-j}{4}}{\frac{\lceil Lj/x\rceil-j}{2}}\frac{\lceil Lj/x\rceil-j}{2}\\&\le\left(\frac{3\sqrt{3}x}{j}\right)^{(\lceil Lj/x\rceil-j)/2}\sqrt{\frac{3}{2}}\frac{e}{2\pi}(\lceil Lj/x\rceil-j)^{1/2}\xrightarrow[j\rightarrow\infty]{}0.
\end{aligned}\]

Hence, given $\epsilon>0$, it exists $j_0$ such that $\forall j\ge j_0$,
\[
\sum_{u=\frac{\lceil Lj/x\rceil-j}{2}}^{\lceil Lj/x\rceil-j}a_u\le \epsilon.
\]

Now we have to deal with $I$, that is, the part $0\le u\le \frac{\lceil Lj/x\rceil-j}{2}$. We treat this part of $s_j$ in a similar way, using again Stirling's approximation. We will distinguish the cases $u$ even and $u$ odd, although the estimate is deduced exactly in the same way. For $u$ even we have
\[\begin{aligned}
a_u&\le \left(\frac{2x}{j}\right)^u\binom{\frac{\lceil Lj/x\rceil-j+u}{2}}{u}\\&\le\frac{e^{1-u}}{\sqrt{2\pi}u!}\left(\frac{\lceil Lj/x\rceil-j+u}{\lceil Lj/x\rceil-j-u}\right)^{1/2}\left(\frac{\lceil Lj/x\rceil-j+u}{\lceil Lj/x\rceil-j-u}\right)^{(\lceil Lj/x\rceil-j-u)/2}\left(\frac{x\lceil Lj/x\rceil}{j}-x+\frac{ux}{j}\right)^u.
\end{aligned}\]

Taking logarithms and using that $\log(1+x)\le x$ we see that
\[e^{-u}\left(\frac{\lceil Lj/x\rceil-j+u}{\lceil Lj/x\rceil-j-u}\right)^{(\lceil Lj/x\rceil-j-u)/2}\le~1.
\]

Moreover, $0\le u\le \frac{\lceil Lj/x\rceil-j}{2}\Rightarrow \lceil Lj/x\rceil-j-u\ge \frac{\lceil Lj/x\rceil-j}{2},$ and
\[
\frac{x\lceil Lj/x\rceil}{j}-x+\frac{ux}{j}\le \frac{3L}{2},
\]
since $\lceil z\rceil\le z+1$. Therefore, we have the following bound, independent of $j$, for $a_u$,
\[
a_u\le\sqrt{\frac{3}{2\pi}}\left(\frac{3L}{2}\right)^u\frac{e}{u!},\ \text{for $u$ even}.
\]

Now we consider the case $u$ odd. Using the same formula, 
\[\begin{aligned}
a_u&\le\left(\frac{2x}{j}\right)^u\binom{\frac{\lceil Lj/x\rceil-j+u-1}{2}}{u}\\&\le \frac{e^{1-u}}{\sqrt{2\pi}u!}\left(\frac{\lceil Lj/x\rceil-j+u-1}{\lceil Lj/x\rceil-j-u-1}\right)^{(\lceil Lj/x\rceil-j-u)/2}\left(\frac{x\lceil Lj/x\rceil}{j}-x+\frac{x(u-1)}{j}\right)^u.
\end{aligned}\]

Again, $e^{-u}\left(\frac{\lceil Lj/x\rceil-j+u-1}{\lceil Lj/x\rceil-j-u-1}\right)^{(\lceil Lj/x\rceil-j-u-1)/2}\le1$, while now, the fact that $\frac{\lceil Lj/x\rceil-j}{2}$ is even and $u$ is odd tells us that
\[\begin{aligned}
&u\le \frac{\lceil Lj/x\rceil-j}{2}-1\Rightarrow \lceil Lj/x\rceil-j-u-1\ge \frac{\lceil Lj/x\rceil-j}{2},
\\&u\le\frac{\lceil Lj/x\rceil-j}{2}+1\Rightarrow u-1\le \frac{\lceil Lj/x\rceil-j}{2},
\\&\frac{x\lceil Lj/x\rceil}{j}-x+\frac{x(u-1)}{j}\le\frac{3L}{2},
\end{aligned}\]
so, therefore
\[
a_u\le\sqrt{\frac{3}{2\pi}}\left(\frac{3L}{2}\right)^u\frac{e}{u!},\ \text{for $u$ even},
\]
and it is clear that
\[
I\le \sum_{u=0}^\infty \sqrt{\frac{3}{2\pi}}\left(\frac{3L}{2}\right)^u\frac{e}{u!}<+\infty.
\]

Hence, by Weierstrass criterion,
\[
\lim_{j\rightarrow\infty}\sum_{u=0}^{\frac{\lceil Lj/x\rceil-j}{2}}a_u=\sum_{u=0}^\infty\lim_{u\rightarrow\infty}a_u=\sum_{u=0}^\infty\left(\frac{L^2}{\pi^2}(1+\cos(\pi x/L))\right)^u\frac{1}{u!}=e^{L^2(1+\cos(\pi x/L))/\pi^2}.
\]

If $j$ or $\lceil Lj/x\rceil$ are not of the form $4n$ with $n$ integer, the proof is the same, we only have to take care of the summation limits in the expression of $s_j$, but once we know this expression, we can follow this argument.

Now we have to apply this procedure to the denominator $t_j$. Assuming again that $\lceil jL/x\rceil\equiv0$ (mod 4), we have that
\[
t_j=\sum_{u=0}^{\lceil Lj/x\rceil}b_u,\ \text{where} 
\]
\[
b_0=1,\ \ b_u=\left(\frac{2x}{j}\right)^u\sum_{l_1=1}^{\frac{\lceil L j/x\rceil}{2}-\left[\frac{u-1}{2}\right]}\ \sum_{l_2=l_1}^{\frac{\lceil L j/x\rceil}{2}-\left[\frac{u-1}{2}\right]}\dots\sum_{l_u=l_{u-1}}^{\frac{\lceil L j/x\rceil}{2}-\left[\frac{u-1}{2}\right]}q_{2l_1-1}q_{2l_2}\dots q_{2l_u+u-2},
\]
for $1\le u\le \lceil Lj/x\rceil$. We can use the same argument we have used above to show that

\[
b_u\xrightarrow[\substack{j\rightarrow \infty}]{}\left(\frac{2L^2}{\pi}\right)^u\int_{0}^{1/2}\int_{x_1}^{1/2}\dots\int_{x_{u-1}}^{1/2}\sin(2\pi x_1)\sin(2\pi x_2)\dots\sin(2\pi x_u)dx_u\dots dx_2dx_1,
\]
and, again,
\[\begin{aligned}\left(\frac{2L^2}{\pi}\right)^u&\int_{0}^{1/2}\int_{x_1}^{1/2}\dots\int_{x_{u-1}}^{1/2}\sin(2\pi x_1)\sin(2\pi x_2)\dots\sin(2\pi x_u)dx_u\dots dx_2dx_1,\\&=\left(\frac{2L^2}{\pi}\int_{0}^{1/2}\sin(2\pi z)dz\right)^u\frac{1}{u!}=\left(\frac{2L^2}{\pi^2}\right)^u\frac{1}{u!},\end{aligned}\]
and this implies that, using again Weierstrass criterion,
\[
\lim_{j\rightarrow\infty}t_j=e^{2L^2/\pi^2}.
\]

Finally, we have
\[
\lim_{j\rightarrow \infty}f_j(x)=\frac{e^{L^2(1+\cos(\pi x/L))/\pi^2}}{e^{2L^2/\pi^2}}=e^{L^2(\cos(\pi x/L)-1)/\pi^2},
\]
exactly the minimizer of the periodic uncertainty principle setting there $h=\frac{\pi}{L}$ and the initial condition $\omega(0)=1$. \qed

\subsection{Periodic case}

In this case we will consider the following symmetric and skew-symmetric operators, represented by the matrices 
\begin{equation}\label{118}
\mathcal{S}_{per}=\left[\begin{array}{ccc}
-Nh&&0\\
&\ddots&\\
0&&Nh
\end{array}\right],\ \ \
\mathcal{A}_{per}=\frac{1}{2h}\left[\begin{array}{ccccc}
0&1&0&\cdots&-1\\
-1&0&1&\cdots&0\\
0&-1&0&\ddots&0\\
 & &\ddots&\ddots&\\
 1&0&\cdots&-1&0\end{array}\right].
\end{equation}

Since the operators, acting over sequences $(u_k)_{k=-N}^N$, are represented by a symmetric and a skew-symmetric matrix respectively, the operators are symmetric and skew-symmetric respectively. 

The commutator $[\mathcal{S}_{per},\mathcal{A}_{per}]$ is represented by the matrix $\mathcal{S}_{per}\mathcal{A}_{per}-\mathcal{A}_{per}\mathcal{S}_{per}$, so we have
\[
[\mathcal{S}_{per},\mathcal{A}_{per}]u_k=\left\{\begin{array}{lc}Nu_N- u_{-N+1}/2,&k=-N,\\-u_{k+1}/2-u_{k-1}/2,&k=-N+1,\cdots,N-1,\\Nu_{-N}-u_{N-1}/2,&k=N,\end{array}\right.
\]
and, after some calculations we have
\[\begin{aligned}
\langle-[\mathcal{S}_{per},\mathcal{A}_{per}]u,u\rangle&=h\Re\sum_{k=-N}^{N-1}u_k\overline{u_{k+1}}-2Nh\Re (u_N\overline u_{-N})\\&=h\sum_{k=-N}^N|u_k|^2-\frac{h^2}{2}h\sum_{k=-N}^{N-1}\left|\frac{u_{k+1}-u_{k}}{h}\right|^2\\&\ \ \ -\frac{h}{2}\left(|u_{-N}|^2+|u_N|^2+4N\Re(u_N\overline{u_{-N}})\right).
\end{aligned}\]

Now we look for the minimizing sequence $\omega=(\omega_{k,N}^h)_{k=-N}^N$ that satisfies the identity in the last equality. For this sequence, $(\mathcal{A}_{per}+\alpha\mathcal{S}_{per})\omega=0$, that is
\begin{equation}\label{eq019}\left\{\begin{array}{lc}(\omega_{-N+1,N}^h-\omega_{N,N}^h)/2h-\alpha Nh\omega_{-N,N}^h=0,\\ (\omega_{k+1,N}^h-\omega_{k-1,N}^h)/2h+\alpha kh\omega_{k,N}^h=0,&k=-N+1,\dots,N-1,\\ (\omega_{-N,N}^h-\omega_{N-1,N}^h)/2h+\alpha Nh\omega_{N,N}^h=0.\end{array}\right.\end{equation}

We can solve this system and write $\omega_{k,N}^h$ in terms of $\omega_{0,N}^h$ using continued fractions. Then, studying the limit of $\omega_{k,N}^h$ when $N$ tends to infinity, we can see that, if we solve the system with the initial condition $\omega_{0,N}^h=\omega_0$,
\[
\omega_{k,N}^h\xrightarrow[N\rightarrow \infty]{}\frac{I_k(1/\alpha h^2)}{I_0(1/\alpha h^2)}\omega_0,
\]
which was the minimizing sequence of our first uncertainty principle. We do not give the details of this here because it is a bit easier to do that in the next case. Then we have the following result:

\begin{teo}\label{teo51} For all $u=(u_k)_{k=-N}^N$
\[\begin{aligned}
&|\langle-[\mathcal{S}_{per},\mathcal{A}_{per}]u,u\rangle|\\&\ \le 2\left(h\sum_{k=-N}^{N}|khu_k|^2\right)^{1/2}\left(h\sum_{k=-N+1}^{N-1}\left|\frac{u_{k+1}-u_{k-1}}{2h}\right|^2+\left|\frac{u_{-N+1}-u_{N}}{2h}\right|^2+\left|\frac{u_{-N}-u_{N-1}}{2h}\right|^2\right)^{1/2},
\end{aligned}\]
and the equality is attained for the sequence $(\omega_{k,N}^h)$ satisfying \eqref{eq019}. Moreover, when we let $N$ tend to $\infty$, this sequence tends to the minimizer of \eqref{eq006}.
\end{teo}

Now we are going to see that we do not have a Virial identity in this finite case. In order to simplify, we set $h=1$.

The equation we consider here is
\[\left\{\begin{array}{ll}\partial_tu_k=i(u_{k+1}-2u_k+u_{k-1}),&k=-N+1,\dots,N-1,\\
\partial_tu_{N}=i(u_{-N}-2u_{N}+u_{N-1}),&\\ \partial_tu_{-N}=i(u_{-N+1}-2u_{-N}+u_{N}).&\end{array}\right.
\]

Differentiating $\sum_{k=-N}^{N}|u_k(t)|^2$ we notice that this quantity is invariant.

Now we differentiate $F(t)=\sum_{k=-N}^{N}k^2|u_k|^2$, getting
\[
\dot F(t)=2\Im\sum_{k=-N+1}^{N}(1-2k)u_k\overline{u_{k-1}}.\]

\[\begin{aligned}
\ddot F(t)&=2\Re\Biggl(-2\sum_{k=-N+1}^{N-1}u_{k+1}\overline{u_{k-1}}+2\sum_{-N+2}^{N}|u_{k-1}|^2\bigr.\\ &\ \bigl.\ \ \ +(2N-1)(u_{-N}\overline{u_{N-1}}-|u_{N}|^2-|u_{-N}|^2+u_{-N+1}\overline{u_N})\Biggr).
\end{aligned}\]

In the classic and $\ell^2(\mathbb{Z})$ cases, $\ddot F(t)=C\ge 0$. Furthermore, $\ddot F(t)$ was 8 times the momentum term on the uncertainty principle. This is not the case of the Periodic case, since
\[
\ddot F(t)=8\left(\sum_{k=-N}^{N}\left|\frac{u_{k+1}-u_{k-1}}{2}\right|^2+(2N+1)\left(\frac{u_{-N}\overline{u_{N-1}}+u_{-N+1}\overline{u_N}-\left|u_{N}\right|^2-\left|u_{-N}\right|^2}{4}\right)\right),
\]
where we make the identification $u_{N+1}=u_{-N}$ and $u_{-N-1}=u_{N}$. As we can see, $\ddot F(t)$ is not a positive constant. We can take another derivative to check that 
\[
\dddot F(t)=C_N\Im(3u_{N}\overline{u_{N-1}}-u_{-N}\overline{u_{N2}}+u_{-N+2}\overline{u_N}-3u_{-N+1}\overline{u_{-N}})\ne 0,
\]
and we also have
\[\begin{aligned}
&N=3,\ \ u(0)=(0,1,0,0,0,1,0)\ \Rightarrow\ \ddot{F}(0)=8\ge 0,\\
&N=3,\ \ u(0)=(2,1,0,0,0,1,0)\ \Rightarrow\ \ddot{F}(0)=-12\le 0.
\end{aligned}\]

\begin{rem} If we had that $\ddot F(t)$ is the momentum term, then we would have that $\dddot{F}(t)$=0.\end{rem}

\subsection{Dirichlet case}

Now we consider the Hilbert space
\[\mathcal{H}_{dir}=\{a=(a_k)_{k=-N}^N: a_N=a_{-N}=0\},\]
and the operators
\begin{equation}\label{eq021}
\mathcal{S}_{dir}=\left[\begin{array}{ccc}
-Nh&&0\\
&\ddots&\\
0&&Nh
\end{array}\right],\ \ \
\mathcal{A}_{dir}=\frac{1}{2h}\left[\begin{array}{ccccc}
0&0&0&\cdots&0\\
-1&0&1&\cdots&0\\
0&-1&0&\cdots&0\\
 & &\ddots&\ddots&\\
 0&0&\cdots&0&0\end{array}\right].
\end{equation}

These operators are the same operators we take in \eqref{018} but with a slight modification in $\mathcal{A}_{dir}$ in  order to send a sequence in $\mathcal{H}_{dir}$ to another sequence in $\mathcal{H}_{dir}$. Thanks to this, both operators acting on sequences in $\mathcal{H}_{dir}$ give another sequence in $\mathcal{H}_{dir}$ and they are respectively symmetric and skew-symmetric.

The uncertainty principle now is very similar to the one we get above, but we have to take into account that the first and the last components of the sequences are zero and then the uncertainty principle is, $\forall u\in \mathcal{H}_{dir}$,
\begin{equation}\label{eq021}\begin{aligned}
&\left|h\sum_{k=-N+1}^{N-1}|u_k|^2-\frac{h^2}{2}h\sum_{k=-N}^{N-1}\left|\frac{u_{k+1}-u_{k}}{h}\right|^2\right|=\left|\Re h\sum_{k=-N+1}^{N-1}u_{k+1}\overline{u_k}\right|\\&\hspace{3cm}\le2\left(h\sum_{k=-N+1}^{N-1}|khu_k|^2\right)^{1/2}\left(h\sum_{k=-N+1}^{N-1}\left|\frac{u_{k+1}-u_{k-1}}{2h}\right|^2\right)^{1/2}. \end{aligned}\end{equation}

Now we want to see who the minimizing sequence is in this inequality. This sequence $\omega=(\omega_{k,N}^h)_{k=-N}^N\in \mathcal{H}_{dir}$, as before, has to satisfy $(\alpha \mathcal{S}_{dir}+\mathcal{A}_{dir})\omega=\bf{0}$, that is, $\omega_{N,N}^h=\omega_{-N,N}^h=0$ and
\[
\alpha kh\omega_{k,N}^h+\frac{\omega_{k+1,N}^h-\omega_{k-1,N}^h}{2h}=0\Longleftrightarrow \omega_{k+1,N}^h+2\alpha kh^2\omega_{k,N}^h=\omega_{k-1,N}^h,\ \ \ k=-N+1,\dots,N-1.
\]

Considering the equation $k=0$, we have that $\omega_{1,N}^h=\omega_{-1,N}^h$, and, by induction, we easily see that $\omega_{-k,N}^h=\omega_{k,N}^h,\ k=-N+1,\dots,N-1$, and, by an iterative process
\begin{equation}\label{eq022}\omega_{k,N}^h=\frac{1}{[2k\alpha h^2,\dots,2(N-1)\alpha h^2]}\frac{1}{[2(k-1)\alpha h^2,\dots,2(N-1)\alpha h^2]}\cdots\frac{1}{[2\alpha h^2,\dots,2(N-1)\alpha h^2]}\omega_{0,N}^h,\end{equation}
where
\[[a_0,a_1,\cdots,a_n]=a_0+\frac{1}{a_1+\frac{1}{\ddots+\frac{1}{a_n}}}.\]

In order to compute the value of each continued fraction, we use again (see Section 5.1) Theorem 149 in \cite{i}, and we observe that
\[
[2k\alpha h^2,\dots,2(N-1)\alpha h^2]=\frac{(-1)^{N+k}K_{k-1}(1/\alpha h^2)I_N(1/\alpha h^2)+I_{k-1}(1/\alpha h^2)K_N(1/\alpha h^2)}{(-1)^{N+k+1}K_k(1/\alpha h^2)I_N(1/\alpha h^2)+I_k(1/\alpha h^2)K_N(1/\alpha h^2)}.
\]

Since we know that $I_N(1/\alpha h^2)$ tends to zero and $K_N(1/\alpha h^2)\simeq C N!$  when $N$ tends to infinity, we have
\[[2k\alpha h^2,\dots,2(N-1)\alpha h^2]\xrightarrow[N\rightarrow\infty]{}\frac{I_{k-1}(1/\alpha h^2)}{I_{k}(1/\alpha h^2)},\]
hence, from \eqref{eq022}, under the assumption that $\omega_{0,N}^h=\omega_0$ for all $N$,
\[
\omega_{k,N}^h\xrightarrow[N\rightarrow\infty]{}\frac{I_{k}(1/\alpha h^2)}{I_{k-1}(1/\alpha h^2)}\frac{I_{k-1}(1/\alpha h^2)}{I_{k-2}(1/\alpha h^2)}\cdots\frac{I_{1}(1/\alpha h^2)}{I_{0}(1/\alpha h^2)}\omega_0=\frac{I_{k}(1/\alpha h^2)}{I_{0}(1/\alpha h^2)}\omega_0.
\]

Therefore we recover the minimizing sequence of the first uncertainty principle we have seen here.

\begin{teo}\label{teo52} For all $u=(u_k)\in \mathcal{H}_{dir}$ the inequality \eqref{eq021} holds, and the equality is attained for the sequence $(\omega_{k,N}^h)$ given by \eqref{eq022}. Moreover, when we let $N$ tend to $\infty$, this sequence tends to the minimizer of \eqref{eq006}.
\end{teo}

Wondering about the existence of an analogue of \eqref{eq002} in this Dirichlet case (we simplify again $h=1$), we consider a solution to the discrete Schrödinger equation
\[\left\{\begin{array}{ll}\partial_tu_k=i(u_{k+1}-2u_k+u_{k-1}),&k=-N+2,\dots,N-2,\\
\partial_tu_{N-1}=i(-2u_{N-1}+u_{N-2}),&\\ \partial_tu_{-N+1}=i(u_{-N+2}-2u_{-N+1}).&\end{array}\right.
\]

It is easy to check that this equation is $\mathcal{H}_{dir}-$invariant.Moreover,
\[
\dot F(t)=2\Im\sum_{k=-N+2}^{N-1}(2k-1)u_k\overline{u_{k-1}}.\]

Taking another derivative, 
\[\begin{aligned}
\ddot F(t)&=2\Re\Biggl(-2\sum_{k=-N+2}^{N-2}u_{k+1}\overline{u_{k-1}}+2\sum_{-N+3}^{N-1}|u_{k-1}|^2-(2N-3)|u_{N-1}|^2\\&\ \ \ -(2N-3)|u_{-N+1}|^2\Biggr).
\end{aligned}\]

In the classic and $\ell^2(\mathbb{Z})$ cases, $\ddot F(t)=C\ge 0$. Furthermore, $\ddot F(t)$ was 8 times the momentum term on the uncertainty principle. This is not the case of the Dirichlet case, since
\[
\ddot F(t)=8\left(\sum_{k=-N+1}^{N-1}\left|\frac{u_{k+1}-u_{k-1}}{2}\right|^2-(2N-2)\left|\frac{u_{N-1}}{2}\right|^2-(2N-2)\left|\frac{u_{-N+1}}{2}\right|^2\right).
\]

Moreover, $\ddot F(t)$ is not a positive constant. We can take another derivative to check that 
\[
\dddot F(t)=C_N\Im(u_{N-2}\overline{u_{N-1}}+u_{-N+2}\overline{u_{-N+1}})\ne 0,
\]
where $C_N$ is a constant which depends on $N$. We also have
\[\begin{aligned}
&N=3,\ \ u(0)=(0,1,0,0,0,1,0)\ \Rightarrow\ \ddot{F}(0)=-12\le 0,\\
&N=3,\ \ u(0)=(0,1,2,0,0,1,0)\ \Rightarrow\ \ddot{F}(0)=4\ge 0.
\end{aligned}\]

\begin{rem} Even if we had that $\ddot F(t)$ is the momentum term, then $\dddot F(t)$ would not be zero, as we can see differentiating the momentum term, being this a  difference between the Dirichlet case and the Periodic case.\end{rem}

\begin{rem} The non-existence of a convex parabola like \eqref{eq003} in these finite cases makes sense, since, as we have said in the introduction, in the continuous case, when the periodic Schr\"{o}dinger equation is considered, there is no equivalent to Theorem \ref{teo11}.\end{rem}

\section{Acknowledgments.} The author is supported by the predoctoral grant BFI-2011-11 of the Basque Government and by MTM2011-24054 and IT641-13. The author would also like to thank O. Ciaurri and L. Vega, without whose help this paper would not have been possible, and the reviewers for their constructive comments that have improved the paper.

\end{document}